\documentclass[oneside, a4paper,11pt,reqno]{amsart}
\usepackage{amssymb,amsmath,amsthm}
\usepackage[T1]{fontenc}

\newcommand{\Z}{{Z\!\!\!Z}}

\newtheorem{hypo}{Hypothesis}

\newtheorem{prop}[hypo]{Proposition}

\newtheorem{thm}[hypo]{Theorem}

\newtheorem{lem}[hypo]{Lemma}

\newtheorem{rqe}[hypo]{Remark}

\DeclareMathOperator{\var}{Var}

\def\R{{\bf R}}

\def\A{\mathcal{A}}
\def\B{\mathcal{B}}
\def\C{\mathcal{C}}
\def\D{\mathcal{D}}

\def\I{\mathcal{I}}
\def\O{\mathcal{O}}
\def\E{\mathcal{E}}
\def\R{\mathcal{R}}
\def\F{\mathcal{F}}

\def\PP{\mathbb{P}}
\def\RR{\mathbb{R}}
\def\ZZ{\mathbb{Z}}

\def\EE{\mathbb{E}}
\def\NN{\mathbb{N}}

\newcommand {\pare}[1] {\left( {#1} \right)}
\newcommand {\cro}[1] {\left[ {#1} \right]}
\newcommand {\refeq}[1] {(\ref{#1})}
\newcommand {\va}[1] {\left| {#1} \right|}
\newcommand {\acc}[1] {\left\{ {#1} \right\}}
\newcommand {\floor}[1] {\left\lfloor {#1} \right\rfloor}

\newcommand {\ceil}[1] {\left\lceil {#1} \right\rceil}

\title[A local limit theorem for RWRS and RWROL]
       {A local limit theorem for random walks in random scenery and on randomly oriented lattices}
\author{Fabienne Castell} 
\address{LATP, UMR CNRS 6632. Centre de Math\'ematiques et Informatique.
Universit\'e Aix-Marseille I. 39, rue Joliot Curie. 13 453 Marseille Cedex
13. France.}
\email{Fabienne.Castell@cmi.univ-mrs.fr}

\author{Nadine Guillotin-Plantard} 
\address{Institut Camille Jordan, CNRS UMR 5208, Universit\'e de Lyon, Universit\'e Lyon 1, 43, Boulevard du 11 novembre 1918, 69622 Villeurbanne, France.}
\email{nadine.guillotin@univ-lyon1.fr}

\author{Fran\c{c}oise P\`ene}
\address{Universit\'e Europ\'eenne de Bretagne, Universit\'e de Brest,
Laboratoire de Math\'ematiques, UMR CNRS 6205, 29238 Brest cedex, France}
\email{francoise.pene@univ-brest.fr}

\author{Bruno Schapira}
\address{D\'epartement de Math\'ematiques, CNRS UMR 8628, B\^at. 425, Universit\'e Paris-Sud 11, F-91405 Orsay, cedex, France. }
\email{bruno.schapira@math.u-psud.fr}

\keywords{Random walk in random scenery; random walk on randomly oriented lattices; local limit theorem; stable process\\
This research was supported by the french ANR projects MEMEMO and RANDYMECA}

\begin{document}

\begin{abstract} 
Random walks in random scenery are processes defined by  
$Z_n:=\sum_{k=1}^n\xi_{X_1+...+X_k}$, where $(X_k,k\ge 1)$ and $(\xi_y,y\in\mathbb Z)$
are two independent sequences of i.i.d. random variables. We assume here that their distributions belong
to the normal domain of attraction of stable laws 
with index $\alpha\in (0,2]$ and $\beta\in (0,2]$ respectively. These processes were first studied by H. Kesten and F. Spitzer, who proved
the convergence in distribution when $\alpha\neq 1$ and as $n\to \infty$, of
$n^{-\delta}Z_n$, for some suitable $\delta>0$ depending on $\alpha$ and $\beta$. 
Here we are interested in the convergence, as $n\to \infty$, of $n^\delta{\mathbb P}(Z_n=\lfloor n^{\delta} 
x\rfloor)$, when $x\in \RR$ is fixed. We also consider the case of random walks on randomly oriented lattices for which we obtain similar results. 
\end{abstract} 
\maketitle

\section{Introduction} 

\subsection{About the model.}
  Random walks in random scenery (RWRS) 
are simple models of processes in disordered
media with long-range correlations. They have been used in a wide
variety of models in physics to study anomalous dispersion in layered
random flows \cite{matheron_demarsily},  diffusion with random sources, 
or spin depolarization in random fields (we refer the reader
to Le Doussal's review paper \cite{ledoussal} for a discussion of these
models). 

  On the mathematical side, motivated by the construction of 
new self-similar processes with stationary increments, 
Kesten and Spitzer \cite{KestenSpitzer} and Borodin  \cite{Borodin, Borodin1} 
introduced RWRS in dimension one and proved functional limit theorems. 
These processes are defined as follows. Let $\xi:=(\xi_y,y\in \ZZ)$ and $X:=(X_k,k\ge 1)$ 
be two independent sequences of independent
identically distributed random variables taking values in $\RR$ and $\ZZ$ 
respectively. The sequence $\xi$ is called the {\it random scenery}. 
The sequence $X$ is the sequence of increments of the {\it random walk}  
$(S_n, n \geq 0)$
defined by $S_0:=0$ and  $S_n:=\sum_{i=1}^{n}X_i$, for $n\ge 1$. 
The {\it random walk in random scenery} $Z$ is 
then defined for all $n\ge 1$ by
$$Z_n:=\sum_{k=0}^{n-1}\xi_{S_k}.$$ 
Denoting by  
$N_n(y)$ the local time of the random walk $S$~:
$$N_n(y)=\#\{k=0,...,n-1\ :\ S_k=y\} \, ,
$$
it is straightforward to see that 
$Z_n$ can be rewritten as $Z_n=\sum_y\xi_yN_n(y)$.

As in \cite{KestenSpitzer}, 
the distribution of $\xi_1$ is assumed to belong to the normal 
domain of attraction of a strictly stable distribution $\mathcal{S}_{\beta}$ of 
index $\beta\in (0,2]$, with characteristic function given by
$$\phi(u)=e^{-|u|^\beta(A_1+iA_2 \text{sgn}(u))}\quad u\in\mathbb{R},$$
where $0<A_1<\infty$ and $|A_1^{-1}A_2|\le |\tan (\pi\beta/2)|$. When $\beta\neq 1$, 
this is the most general form of a strictly stable distribution. In the case 
$\beta=1$, this is the general form of a random variable $Y$ with strictly stable 
distribution satisfying the following symmetry condition~: 
\begin{equation}
 \sup_{M>0}|\EE(Y {\bf 1}_{\{|Y|<M\}})| <+\infty .
\end{equation}
We will denote by $f_\beta$ the density function of the law 
$\mathcal{S}_{\beta}$. 

\noindent Concerning the random walk, 
the distribution of $X_1$ is assumed to belong to the normal 
domain of attraction of a strictly stable distribution 
$\mathcal{S}_{\alpha}$ with 
index $\alpha\in (0,2]$. In this paper we will actually not consider the case 
$\alpha=1$ (see Remark 2 in \cite{KestenSpitzer} for some discussion 
on this case). 

\noindent 
\noindent Then the following weak convergences hold in the space of 
c\`ad-l\`ag real-valued functions 
defined on $[0,\infty)$ and on $\mathbb R$ respectively~:
$$\left(n^{-\frac{1}{\alpha}} S_{\lfloor nt\rfloor}\right)_{t\geq 0}   
\mathop{\Longrightarrow}_{n\rightarrow\infty}
^{\mathcal{L}} \left(U(t)\right)_{t\geq 0}$$
$$\mbox{\rm and} \  \  \   \left(n^{-\frac{1}{\beta}} 
\sum_{k=0}^{\lfloor nx\rfloor}\xi_k\right)_{x\in\mathbb R}
   \mathop{\Longrightarrow}_{n\rightarrow\infty}^{\mathcal{L}} \left(Y(x)\right)_{x\in\mathbb R},$$
where $U$ and $Y$ are two independent L\'evy processes such 
that $U(0)=0$, $Y(0)=0$, 
$U(1)$ has distribution $\mathcal{S}_\alpha$, $Y(1)$ and 
$Y(-1)$ have distribution  $\mathcal{S}_\beta$.
When $\alpha\in (1,2]$, the random walk $(S_n, n \geq 0)$ is recurrent, 
and the limiting process $U$ admits a local time process. We denote
 by $(L_t(x), t \in \RR^+, x \in \RR)$ 
the jointly continuous version of this  local time. 


\noindent 
Let $$\delta:=1-\frac{1}{\alpha}+\frac{1}{\alpha \beta}.$$
Papers \cite{KestenSpitzer,Borodin,Borodin1} 
  proved that the following weak convergences 
hold in the space of continuous real-valued functions defined on  $[0,\infty)$~:  
\begin{eqnarray}
\label{eq1.08}
\mbox{if} \ \alpha>1,\ &\ & \left(n^{-\delta} Z_{nt}\right)_{t\geq 0} \mathop{\Longrightarrow}
_{n\rightarrow\infty}^{\mathcal{L}}  \left(\Delta(t)\right)_{t\geq 0}\\
\mbox{if}\  \alpha<1,\ &\ & \left(n^{-\frac 1\beta} Z_{nt}\right)_{t\geq 0}
 \mathop{\Longrightarrow}
_{n\rightarrow\infty}^{\mathcal{L}}   \left(Y(t)\ {\mathbb E}[(\widetilde N_{\infty}^{\beta-1}(0)]^{\frac 1\beta} 
     \right)_{t\geq 0}\, ,
\end{eqnarray} 
where 
\begin{itemize} 
\item $Z_s$ is defined as the linear interpolation
$Z_s=Z_n+(s-n)(Z_{n+1}-Z_n)$ when $n\leq s\leq n+1$, 
\item $\Delta$ is the 
process defined by
$$\Delta(t)=\int_{-\infty}^{+\infty} L_t(x)\, {\rm d}Y(x)\, ,$$
\item $\widetilde N_\infty(0)$ is the total time spent in 0 by 
the two-sided random walk
$(S_k,k\in{\mathbb Z})$ with $S_{-k}=-\sum_{m=1}^{k}X_{-m}$
(where $(X_{-k},k\ge 1)$ is independent of $(X_k,k\ge 1)$ and
with the same distribution).
\end{itemize} 
The limiting process $\Delta$ is known to be a  
continuous $\delta$-self-similar 
process with stationary increments. 
It can be seen as a mixture of $\beta$-stable processes, but it is not a stable process. 

  Since these seminal papers, RWRS have been extensively studied. Far from being
exhaustive, we can cite limit theorems in higher dimension \cite{bolthausen}, 
strong approximation results and laws of the iterated 
logarithm \cite{KhoshnevisanLewis,CsakiKonigShi,CR}, limit theorems for
 correlated sceneries or walks \cite{GuillotinPrieur,CohenDombry},
large and moderate 
deviations results 
\cite{CastellPradeilles,Castell,AsselahCastell,GantertKonigShi}. Our contribution in this paper is a local version of the convergence results from \cite{KestenSpitzer}, as we make more precise in the next subsection.

\subsection{The results.} 
Our first statement is obtained in the case when the 
$\xi'$s are $\ZZ$-valued random variables.
Let $\varphi_\xi(u):=\EE[e^{iu\xi_1}]$ be the characteristic 
function of $\xi_1$. Remember that there exists an integer $d\ge 1$ 
such that $\{u\ :\ |\varphi_\xi(u)|=1\}=\frac{2\pi}d{\mathbb{Z}}$ 
($d$ is the g.c.d. of the set of $b-c$ where $b$ and $c$
belong to the support of the distribution of 
$\xi_1$)\footnote{Note that $\xi$ is said to be non-arithmetic if $d=1$.}.

\noindent Our first result concerns the case $\alpha>1$~:  
\begin{thm} {\bf Lattice case, $\alpha >1$.}
 \\
\label{thmTLL}
Assume that $\alpha \in (1,2]$ and $\beta\in (0,2]$.  
Let $C(x)$ be the continuous function defined by 
$$C(x):=  \EE\left[|L|_\beta^{-1}f_\beta(|L|^{-1}_\beta x)\right] \quad 
\text{for all }x\in \RR,$$
where $|L|_\beta:= \left(\int_\RR L_1^\beta(y)\, dy\right)^{1/\beta}$. 
Then, for every $x \in \RR$, we have  $0 < C(x) <\infty$ and

$\bullet$ if ${\mathbb P}\left(n\xi_1-\floor{n^\delta x}\notin d{\mathbb Z}\right)=1$, 
then ${\mathbb P} \pare{Z_n= \floor{n^{\delta}x}}=0$;

$\bullet$ if ${\mathbb P}\pare{n\xi_1-\floor{n^\delta x}\in d{\mathbb Z}}=1$, 
then
$$ {\mathbb P} \pare{
Z_{n}=\floor{n^\delta x}}= d \frac{C(x)}{n^{\delta}}+o(n^{-\delta}) \, ,
$$
where the $o(n^{-\delta})$ is uniform in $x$.
\end{thm}

\noindent \textbf{Remark.} There is no other alternative for the law of $\xi_1$.
Indeed, let $b$ be in the support of $\xi_1$.
Then $n\xi_1$ belongs to $nb+d\mathbb Z$.
Hence the condition $n\xi_1-\lfloor n^\delta x\rfloor\in d\mathbb Z$ 
is equivalent to $\lfloor n^\delta x\rfloor-nb\in d\mathbb Z$.

\noindent Our second result concerns the case $\alpha<1$~: 
\begin{thm} {\bf Lattice case, $\alpha <1$.}\\
\label{thmTLL2}
Assume that $\alpha \in (0,1)$, $\beta\in (0,2]$ and $x\in{\mathbb R}$. 
Let $D(x):=rf_\beta(rx)$, with $r:=\EE[\widetilde N_\infty^{\beta-1}(0)]^{-1/\beta}$. Then  
 
$\bullet$ if ${\mathbb P}\pare{ n\xi_1-\floor{n^{\frac{1}{\beta}} x}
\notin d{\mathbb Z}} =1$, 
then ${\mathbb P} \pare{Z_n= \floor{n^{\frac{1}{\beta}}x}}=0$;

$\bullet$ if ${\mathbb P}\pare{n\xi_1-\floor{n^{\frac{1}{\beta}} x}\in d{\mathbb Z}}=1$, 
then
$$ {\mathbb P} \pare{
Z_{n}=\floor{n^{\frac{1}{\beta}} x}}=d \frac{D(x)}{n^{\frac{1}{\beta}}}+o(n^{-\frac{1}{\beta}}) \, ,
$$
where the $o(n^{-\frac{1}{\beta}})$ is uniform in $x$;
\end{thm}

\noindent Finally we get the local limit theorem when 
$\xi$ is strongly nonlattice, i.e. when $\displaystyle\limsup_{\vert u\vert\rightarrow +\infty}
\vert \varphi_\xi(u)\vert<1$.

\begin{thm} \label{nonlattice} {\bf Strongly nonlattice case.}
\begin{itemize} 
\item If $\alpha > 1$ and $\beta \in (0,2]$, then for all $a$, $b\in \RR$ such that $a < b$, 
\[ \lim_{n \rightarrow \infty} 
n^{\delta}  \PP\cro{Z_{n} \in [{n^\delta x} +a; n^\delta x + b]}=
C(x) (b-a) \, .
\] 
\item If $\alpha < 1$ and $\beta \in (0,2]$, then for all $a$, $b\in \RR$ such that $a < b$,
\[ \lim_{n \rightarrow \infty} n^{\frac 1 \beta} 
\PP\cro{Z_{n} \in [ n^{\frac 1 \beta  } x +a; n^{\frac 1 \beta}  x + b]}=
D(x) (b-a) \, .
\]
\end{itemize} 
\end{thm}

\noindent On the one hand, these results give some qualitative information about the behaviour of $Z$. For instance the transience of the process $Z$ is easily deduced (with Borel-Cantelli Lemma) when $\beta<1$. Note that since $Z$ is not a Markov chain, the recurrence property when $\beta> 1$ does not directly follow from the above local limit theorems. However this can be proved by using an argument from ergodic theory (see \cite{KS}). Indeed, it is enough to remark that when $\beta\in (1,2]$, the random variables $\xi_{S_k}, k\in\mathbb{N}$ form an ergodic and stationary sequence of integrable and centered random variables.

\noindent On the other hand this work was motivated by the study of random walks on randomly oriented lattices. In the simplest case, one should think to the simple random walk defined on a random sublattice of the oriented lattice $\ZZ^2$, which is constructed as follows. On each horizontal line, one removes all edges oriented to the right with probability $1/2$ or those oriented to the left with probability $1/2$, and so independently on each level. Then it is known, and not difficult to see, that the first coordinate of the resulting random walk is closely related to a random walk in random scenery $Z=\sum_k \xi_{S_k}$, with $S$ the simple random walk on $\ZZ$ and the $\xi_y$ i.i.d random variables with geometric distribution (see Section 5 or \cite{GPLN2} for more explanations). 
In \cite{GPLN2} it was conjectured that the probability of return to the origin of this random walk is equivalent to a constant times $n^{-5/4}$. Here we prove a local limit theorem for even more general random walks, giving in particular a proof of this conjecture.  We refer the reader to Section 5 for more precise statements of our results.

\subsection{ Outline of the proof.} Let us give a very rough description 
of the proofs for RWRS. To fix ideas, we do it for $x=0$ and $\alpha >1$. By Fourier
inverse transform, we have to study the asymptotic behavior of 
\begin{equation}
\label{integ}
 \int \EE\cro{e^{itZ_n}} \, dt = \int \EE \cro{\prod_{y \in \Z} \varphi_{\xi}(t
N_n(y))} \, dt \, .
\end{equation} 
For $t$ such that $t N_n(y)$ is small, only the behavior of $\varphi_{\xi}$
around 0 is relevant. Therefore, for $|t| \leq (\sup_y N_n(y))^{-1} \simeq
n^{-1+1/\alpha}$,  
$$\EE \cro{\prod_{y \in \Z} \varphi_{\xi}(t N_n(y))}
\simeq \EE \cro{\exp(- |t|^{\beta} \sum_y N_n(y)^{\beta} (A_1+iA_2\text{sgn}(t)))}.$$
Now, $\sum_y N_n(y)^{\beta}$ is of order $n^{\beta \delta}$, and 
a change of variable $t \leadsto n^{\delta} t$ leads to the 
dominant part in the integral \refeq{integ}. 

For $t \geq (\sup_y N_n(y))^{-1} \simeq
n^{-1+1/\alpha}$, the behavior of $\varphi_{\xi}$ away from $0$ comes 
into play. 
In the strongly nonlattice case, one can find $\epsilon_0 >0$ and $\rho \in (0,1)$
such that $|\varphi_{\xi}(t)| \leq \rho$ for $|t| \geq \epsilon_0$, so that
for $|t| \geq n^{-1+1/\alpha}$, 
\[ 
\va{\prod_{y \in \Z} \varphi_{\xi}(t N_n(y))}
\leq \rho^{\# \acc{y; N_n(y) \geq \frac{\epsilon_0}{t}}}
\leq  \rho^{\# \acc{y; N_n(y) \geq \epsilon_0 n^{1-1/\alpha}}}
\, .
\] 
It is easily seen that there is a large number of points visited at least
$n^{1-1/\alpha}$ times, leading to the result.

The lattice case is more delicate, since in this case $|\varphi_{\xi}(t)|
=1$ for $t \in \frac{2\pi}{d} \ZZ$, so that the inequality 
$\va{\varphi_{\xi}(tN_n(y)) } \leq \rho$ is only valid for the $y$
such that $d(tN_n(y);  \frac{2\pi}{d} \ZZ) \geq \epsilon_0$. Thus, the main
difficulty is to show that for $|t| \geq n^{1-1/\alpha}$, there are
a lot of such sites. This is done by a surgery on the trajectories of the
random walk. 

\noindent Let us briefly describe now the organization of the paper. 
In the next section, we prove Theorem \ref{thmTLL}. 
In Sections 3 and 4, we sketch the proofs of Theorem \ref{thmTLL2}  and 
Theorem \ref{nonlattice}
which are  easier and follow the same lines. In Section 5, the local limit theorem for random walks evolving on randomly oriented lattices is obtained by using similar techniques as for the proof of Theorem \ref{thmTLL}.
Finally in the appendix, we prove some auxiliary results 
on the range of the random walk $S$, that we should need, 
but which could also be of independent interest. 

\section{Lattice case, $\alpha>1$: Proof of Theorem \ref{thmTLL}}
\subsection{Finiteness of $C(x)$.}
\begin{lem} 
For all $x \in \RR$,  $0 < C(x) <+\infty$.
\end{lem}
\begin{proof}
Let $x\in\mathbb R$.
Since $\int_\RR L_1(y)\, dy=1$ and $\beta\le 2$, we have
a.s. $\int_\RR L_1^\beta(y)\, dy \le 1+\sup_y L_1(y)^{(\beta-1)_+}$. 
Hence $\int_\RR L_1^\beta(y)\, dy$ is a.s. finite. So $C(x) > 0$.

\noindent Let us prove now that $C(x)$ is finite. First we have 
$$ C(x) \le  \Vert f_\beta\Vert_\infty \EE[\vert L\vert_\beta^{-1}].$$
Let us assume now that $\beta > 1$. By H\"older's inequality, 
$$1 = \int_\RR L_1(y) \, dy \leq |L|_\beta \left(\int_\RR {\bf 1}(L_1(y) >0) \, dy
\right)^{1-\frac 1\beta}.$$ 
Thus by using Jensen's inequality we get
\begin{eqnarray*}
C(x)&\le& \Vert f_\beta\Vert_\infty
\EE\cro{ \left(\int_\RR {\bf 1}(L_1(y) >0) \, dy\right)^{1-\frac 1\beta}} \\
&\le& \Vert f_\beta\Vert_\infty
\left(\EE\cro{ \left(\int_\RR {\bf 1}(L_1(y) >0) \, dy\right)} \right)^{1-\frac 1\beta}
= \Vert f_\beta\Vert_\infty
\left(\EE[\lambda(U([0,1]))]\right)^{1-\frac 1\beta},
\end{eqnarray*} 
where $\lambda$ denotes the Lebesgue measure on $\RR$ and $U([0,1])$ the set of points visited by $U$ before time 1.
This finishes the proof in the case $\beta >1$, 
since the last quantity is finite (see for example \cite{LGR} p.703).

\noindent Next, if $\beta=1$, then $\vert L\vert_\beta=1$ and $C(x)=f_\beta(x)<+\infty$.

\noindent Assume finally that $\beta < 1$. Then
\[ 1 = \int_{\mathbb{R}} L_1(y) \, dy \leq |L|_{\beta}^{\beta} \pare{\sup_x L_1(x)}^{1-\beta}
\, ,
\]
so that 
\[ \EE \cro{|L|_{\beta}^{-1}} 
\leq \EE \cro{\pare{\sup_x L_1(x)}^{\frac{1-\beta}{\beta}}} 
= \frac{1-\beta}{\beta} \int_0^{+\infty} t^{\frac{1}{\beta}-2} 
\PP \cro{ \sup_x L_1(x) \geq t} \, dt \, .
\]
Therefore it suffices to prove that there exists a constant $c > 0$ such
that  
\begin{equation}
\label{BSL*}
\PP \cro{\sup_x L_1(x) \geq t} \leq 2\exp(-ct) \quad \text{for all }t>0\, .
\end{equation} 
This follows from stronger results proved in \cite{Lacey}, but for sake of completeness, let us give a soft argument here.
For $a>0$, let $\tau_a := \inf \acc{t\ :\ \sup_x L_t(x) \geq a}$. The random variable $\tau_a$ is a stopping time, and
by continuity of $t \mapsto \sup_x L_t(x)$, $\sup_x L_{\tau_a}(x) =a$ 
on $\{\tau_a < \infty\}$. It follows then from the inequality 
\[ \sup_x L_{t+s}(x) \leq \sup_x L_t(x) + \sup_x (L_{t+s}(x)-L_{t}(x))\, ,
\]
 and from the strong Markov property,
 that  for any $a>0$ and $b >0$, 
\[
\PP\cro{\sup_x L_1(x) \geq a +b} 
= \PP \cro{\tau_a \leq 1\, ; \sup_x L_1(x) \geq a +b}
\leq \EE \cro{{\bf 1}_{\{\tau_a \leq 1\}} \,  
\PP_{U_{\tau_a}} \cro{ \sup_x L_1(x) \geq b}}
\,,
\]
where for any $v$, $\PP_v$ denotes the law of the process $U$ starting from $v$. 
By translation invariance, the law of  $\sup_x L_1(x)$ does not depend 
on the starting point of $U$. Therefore, for
any $a>0$ and $b >0$,
\begin{equation} 
\label{SousaddL*}
\PP\cro{\sup_x L_1(x) \geq a +b} 
\leq \PP\cro{\tau_a \leq 1} \PP\cro{\sup_x L_1(x) \geq b}
= \PP\cro{\sup_x L_1(x) \geq a} \PP\cro{\sup_x L_1(x) \geq b} \, .
\end{equation} 
Let $M >0$ be a median of  $\sup_x L_1(x)$. By \refeq{SousaddL*}, for all $t > 0$, 
\[ 
\PP\cro{\sup_x L_1(x) \geq t} \leq \PP\cro{\sup_x L_1(x) \geq M}^{\floor{t/M}} 
\leq \pare{\frac{1}{2}}^{\floor{t/M}} \, ,
\]
which ends the proof of \refeq{BSL*}.   
\end{proof}

\subsection{A first reduction.}

\begin{lem}
\label{formule1}
Let $n\ge 1$ and $x \in \ZZ$ be given.

$\bullet$ If $\PP\cro{n\xi_1-x\notin d{\mathbb Z}}=1$, then ${\mathbb P}(Z_n=x)=0$.

$\bullet$ If $\PP\cro{n\xi_1-x\in d{\mathbb Z}}=1$, then 
$${\mathbb P}(Z_n=x) = \frac d{2\pi}\int_{-\frac{\pi}d}^{\frac\pi d}
\exp(-itx) {\mathbb E}\left[\prod_{y\in{\mathbb Z}}\varphi_{\xi}(tN_n(y))\right]
\, dt \, . $$
\end{lem}
\begin{proof}
We have
$$ {\mathbb P}(Z_n=x)=\frac 1{2\pi}\int_0^{2\pi} \exp(-itx) \varphi_n(t)\, dt\, ,$$
where $\varphi_n$ is the characteristic function of $Z_n$
given by 
$$\varphi_n(t):=
{\mathbb E}\left[\prod_{y\in{\mathbb Z}}\varphi_{\xi}(tN_n(y))\right]\quad \text{for all } t\in \RR.$$
Notice that $e^{\frac{2i\pi\xi_1}d}={\mathbb E}[e^{\frac{2i\pi\xi_1}d}]$ almost surely.
Hence, for any integer $m\ge 0$ and any $u\in \RR$,
$$\varphi_\xi\left(\frac{2m\pi}d+u\right)
=\varphi_\xi\left(\frac{2\pi}d\right)^m\varphi_\xi(u).$$
Therefore 
\begin{eqnarray*}
 {\mathbb P}(Z_n=x)
&=&\frac 1{2\pi} \sum_{k=0}^{d-1} \int_{-\frac\pi d}^{\frac\pi d} 
\exp \pare{ -i \pare{t+\frac{2k\pi}d} x} \,  \varphi_n\left(\frac{2k\pi}d+t\right)\, dt\\
&=&\frac 1{2\pi}\sum_{k=0}^{d-1}\int_{-\frac\pi d}^{\frac\pi d}
\exp(-itx) \exp \pare{-i \frac{2k\pi}d x}
\, {\mathbb E}
\left[\prod_y\left\{\varphi_\xi\left(\frac{2\pi}d\right)^{kN_n(y)}
\varphi_\xi(tN_n(y))\right\}\right]\,
dt\\
&=&\frac 1{2\pi}\left(\sum_{k=0}^{d-1} \exp\pare{-i \frac{2k\pi}d x}
 \varphi_\xi\left(\frac{2\pi}d\right)^{kn}\right)\int_{-\frac\pi d}^{\frac\pi d}
\exp(-itx) \varphi_n(t)\, dt,
\end{eqnarray*}
since $\sum_yN_n(y)=n$.
Moreover, $\left[e^{-i \frac{2\pi}d x} \varphi_\xi\left(\frac{2\pi}d\right)^{n}\right]^{d}
=e^{-i 2\pi x} e^{2i\pi n\xi_1}=1$, thus $e^{-i \frac{2\pi}d x} \varphi_\xi\left(\frac{2\pi}d\right)^{n}$ is a
$d^{th}$ root of the unity.
Hence
\begin{eqnarray*}
\sum_{k=0}^{d-1} e^{-i \frac{2k\pi}{d} x }\varphi_\xi\left(\frac{2\pi}d\right)^{kn}
= \left\{ \begin{array}{ll} 
        d & \textrm{if } \varphi_\xi\left(\frac{2\pi}d\right)^{n} 
e^{-i \frac{2\pi}{d} x}=1 , \\ 
       0 & \textrm{otherwise.} 
\end{array} 
\right. 
\end{eqnarray*} 
Since $\varphi_\xi\left(\frac{2\pi}d\right)=e^{\frac{2i\pi\xi_1}d}$ a.s., the lemma follows. 
\end{proof}
\subsection{The event $\Omega_n$.}
\label{sec:omega_n} 
Set
\[ N_n^* := \sup_y N_n(y) \quad \textrm{and} \quad R_n := \#\{y\ :\ N_n(y)>0\} \, .
\]
\begin{lem}\label{lem:omega_n}
For every $n\ge 1$ and $\gamma > 0$, set
\[
\Omega_n=\Omega_n(\gamma) := \acc{   
R_n \le n^{\frac 1 \alpha+\gamma}\quad \textrm{and} \quad \sup_{y \ne z} 
\frac{\vert N_n(y)-N_n(z)\vert}{|y-z|^{\frac{\alpha-1}{2}}} 
\leq n^{(1-\frac 1 \alpha +\gamma)/2}}  \, .
\]
Then $\PP(\Omega_n) =1 - o(n^{-\delta})$.
Moreover, given $\eta\ge \gamma\max (\alpha/2,2(\beta-1)/\beta)$, the following also holds on $\Omega_n$: 
\begin{eqnarray}
\label{minVn}
N_n^*\le n^{1-\frac 1\alpha+\eta}\ \ \mbox{and}\ \ 
V_n:=\sum_z N_n^\beta(z) \geq \left\{ \begin{array}{ll} 
                                        n^{\delta \beta - \frac{\eta\beta}{2}} & \textrm{if }\beta>1 \\
                                        n^{\delta \beta - \eta(1-\beta)} & \textrm{if }\beta\le 1.
                                     \end{array} 
\right. 
\end{eqnarray} 
\end{lem}

\begin{proof}
We prove in the appendix that for every $\gamma >0$, there exists $C>0$ 
such that  
 \[
 {\mathbb P}\left(
R_n\le \EE[R_n]n^{\gamma}\right)=1-\O(e^{-C n^{\gamma}}).
\]
Since there exists $c>0$ such that $\EE[R_n]\sim c n^{\frac 1\alpha}$
(see \cite{spitzer} p.36),
we conclude that 
\[
\PP(R_n \le n^{\frac 1 \alpha+\gamma})=1-o(n^{-\delta}).
\]

\noindent Now let us prove that 
$${\mathbb P}\left(\sup_{y\ne z} 
\frac{\vert N_n(y)-N_n(z)\vert}{|y-z|^{\frac{\alpha-1}2}}\ge \sqrt{n^{1-\frac 1 \alpha
 +\gamma}}\right)=o(n^{-\delta}).$$
According to the proof of Proposition 5.4 in \cite{LGR}, we have~:
${\mathbb E}[\vert S_n\vert^p]=\O(n^{\frac p\alpha})$, for all $p\in(1,\alpha)$. 
Then Doob's inequality gives that, for all $\delta'>\delta/p$,
$${\mathbb P}(\sup_{k=1,...,n}\vert S_k\vert\ge n^{\frac 1\alpha+\delta'})=\O(n^{-p\delta'})=o(n^{-\delta}).$$
So we can restrict ourselves to the set 
$A_n:=\{\sup_{k=1,...,n}\vert S_k\vert<n^{\frac 1\alpha+\delta'}\}$. 
But on $A_n$, if $N_n(z)>0$ then necessarily $z\in (-n^{\frac 1\alpha+\delta'},n^{\frac 1\alpha+\delta'})$. Thus 
\begin{eqnarray}
\label{eq1} 
{\mathbb P}\left(\sup_{y,z} 
\frac{\vert N_n(y)-N_n(z)\vert}{|y-z|^{\frac{\alpha-1}2}}\ge\sqrt{n^{1-\frac 1 \alpha +\gamma}};A_n\right)
\le 5n^{\frac 2\alpha+2\delta'}\sup_{y\ne z}{\mathbb P}\left(\frac
{\vert N_n(y)-N_n(z)\vert}{|y-z|^{\frac{\alpha-1}2}}\ge \sqrt{n^{1-\frac 1 \alpha +\gamma}}\right).
\end{eqnarray}
Moreover the Markov inequality gives for all $m\ge 1$:
\begin{eqnarray}
\label{eq2}
 {\mathbb P}\left(\frac{\vert N_n(y)-N_n(z)\vert}{|y-z|^{\frac{\alpha-1}2}}
 \ge \sqrt{n^{1-\frac 1 \alpha +\gamma}}\right)
\le\frac{{\mathbb E}[|N_n(y)-N_n(z)|^{2m}]}{|y-z|^{(\alpha-1)m}n^{\left(1-\frac 1 \alpha +\gamma\right)m}} \quad \text{for all }y\neq z.
\end{eqnarray}
In addition, according to \cite{Jain-Pruitt} (see the formula in the middle of page 77, 
with $m=\O(n)$, $a_m^{-1}=\O(n^{-1/\alpha})$
and $Q(z)^{-1}=\O(z^\alpha)$), we have for all $m\ge 1$,
\begin{eqnarray}
\label{eq3}
\sup_{y\ne z} \frac{{\mathbb E}[|N_n(y)-N_n(z)|^{2m}]}
         {|y-z|^{(\alpha-1)m}} =\O(n^{\left(1-\frac 1\alpha)m\right)}).
\end{eqnarray}
Thus if we take  $m>(\delta+2/\alpha+2\delta')/\gamma$, then by using \eqref{eq1}, \eqref{eq2} and \eqref{eq3}, we get 
\[
 {\mathbb P}\left(\sup_{y\ne z} 
\frac{\vert N_n(y)-N_n(z)\vert}{|y-z|^{\frac{\alpha-1}2}}\ge \sqrt{n^{1-\frac 1 \alpha +\gamma}}\right)
=\O\left(\frac{n^{\frac 2\alpha+2\delta'}}{n^{\gamma m}}\right)=o(n^{-\delta}).
\]
We now prove \eqref{minVn}, starting with the upper bound for $N_n^*$. For this let $y_0$ be 
such that $N_n(y_0)=N_n^*$,
and let $z_0$ be the closest point to $y_0$ such that
$N_n(z_0)=0$. Then on $\Omega_n$, 
$$\vert y_0-z_0|\le
R_n\le n^{\frac 1\alpha+\gamma},$$
and thus
\begin{equation}\label{sec:majoN_n*}
 N_n(y_0)\le\sqrt{
   \vert y_0-z_0|^{\alpha-1}n^{1-\frac 1\alpha+\gamma}}
    \le \sqrt{ n^{\left(\frac 1\alpha+\gamma\right)(\alpha-1)}
n^{1-\frac 1\alpha+\gamma}}
     = n^{1-\frac 1\alpha+\frac{\alpha\gamma}2}.
\end{equation}
The desired upper bound for $N_n^*$ follows if
$\eta\ge \alpha\gamma/2$.

To prove the lower bound for $V_n$, we use the fact that
$n = \sum_y N_n(y)$. When $\beta>1$, this gives by using H\"older's inequality:
\[
n \leq \left(\sum_z N_n^\beta(z)\right)^{\frac 1 \beta} R_n^{1-\frac 1 \beta}
\le   (V_n)^{\frac 1\beta}n^{\left(\frac 1\alpha +\gamma\right)
\left(1-\frac 1\beta\right)}.
\]
Hence $V_n^{\frac 1\beta}\ge n^{\delta -\gamma\frac{\beta-1}\beta}$,
and the desired lower bound for $V_n$ follows if $2(\beta-1)\gamma\le \eta\beta$. 
When $\beta\le 1$, we write 
$$n=\sum_y N_n(y) \le V_n (N_n^*)^{1-\beta},$$
and the desired lower bound follows from the upper bound for $N_n^*$ proved just above. 
\end{proof}

%
%
\subsection{Scheme of the proof.}\label{sec:scheme}
Let $\eta>0$. Set $\gamma:=\eta\beta/2$. 
We observe that $\gamma\le\eta$ and that \eqref{minVn} holds with
this choice of $(\eta,\gamma)$. We also set 
\begin{eqnarray*}
\overline{\eta}:=\left\{ \begin{array}{ll} \eta & \text{if } \beta\ge 1\\
                         \eta/\beta & \text{if } \beta <1.
\end{array} 
\right. 
\end{eqnarray*} 
By Lemmas \ref{formule1} and \ref{lem:omega_n}, we have to estimate 
$$
\frac d{2\pi}\int_{-\frac{\pi}d}^{\frac\pi d}
e^{-it \floor{n^{\delta} x}}
{\mathbb E}\left[\prod_y \varphi_\xi(tN_n(y)){\bf 1}_{\Omega_n}\right]
\, dt\, .
$$
This is done in several steps 
presented in the following propositions. 

\begin{prop}\label{lem:equivalent} 
Let $\eta \in \left(0,\frac 1{2\alpha(\beta+1)}\right)$. Then, we have
\[
\frac d {2\pi} \int_{|t| \le n^{-\delta+ \overline{\eta}}}
e^{-it \floor{n^\delta x}}  \EE \cro{\prod_y \varphi_{\xi}(tN_n(y))
   {\bf 1}_{\Omega_n} } \, dt 
= d\frac{C(x)}{n^\delta} + o(n^{- \delta}) \, ,
\]
uniformly in $x\in\mathbb R$.
\end{prop}
Recall next that the characteristic function $\phi$ of the stable distribution ${\mathcal S}_\beta$ 
has the following form~: 
$$\phi(u)=e^{-|u|^\beta(A_1+iA_2 \text{sgn}(u))},$$
for some $0<A_1<\infty$, $|A_1^{-1}A_2|\le |\tan (\pi\beta/2)|$. 
It follows that the characteristic function $\varphi_\xi$ of $\xi_1$ satisfies: 
\begin{eqnarray}
\label{phi0}
1-\varphi_\xi(u)\sim |u|^\beta(A_1+iA_2 \text{sgn}(u)) \quad \textrm{when } u\to 0.
\end{eqnarray}
Therefore there exist constants $\varepsilon_0>0$ and $\sigma>0$ such that
\begin{eqnarray}
\label{majorationphi} 
\max(\vert\phi(u)\vert,\vert\varphi_\xi(u)\vert)\le
\exp\left(-\sigma|u|^\beta\right) \quad \textrm{for all }u\in [-\varepsilon_0,\varepsilon_0].
\end{eqnarray}

\noindent Since $\overline{\varphi_{\xi}(t)} =
 \varphi_{\xi}(-t)$ for every $t\ge 0$, the  following propositions 
achieve the proof of Theorem \ref{thmTLL}:

\begin{prop}\label{sec:step1}
Let $\eta$ be as in Proposition \ref{lem:equivalent}. 
Then there exists $c>0$ such that
$$
\int_{n^{-\delta + \overline{\eta}}}^{\varepsilon_0 n^{-1+\frac{1}{\alpha}-\eta}}
 {\mathbb E}\left[\prod_y \vert \varphi_\xi(tN_n(y))\vert {\bf 1}_{\Omega_n}\right]
\, dt= o(e^{-n^c}).$$
\end{prop}

\begin{prop}\label{sec:step2}
Let $\eta$ be as in Proposition \ref{lem:equivalent}
and let $\varepsilon\in \left(\eta,\frac{\alpha-1}
{\alpha(3+2\beta(\alpha-1))}\right)$ be given. Then there exists $c>0$ such that
\[
\int_{\varepsilon_0 n^{-1+\frac{1}{\alpha}-\eta}}^{n^{-1+\frac{1}{\alpha}+\varepsilon}}
{\mathbb E}\left[\prod_y \vert\varphi_\xi(tN_n(y))\vert
{\bf 1}_{\Omega_n}\right] \, dt
=o(e^{-n^c}) \, .
\]
\end{prop}

\begin{prop}\label{sec:step3}
Let $\eta$ be such that $\gamma<\min\left(\frac 1{2\alpha^2},\frac 12\frac{\alpha-1}\alpha\right)$
and let $\varepsilon\in \left((\frac{2\alpha}\beta+1)\gamma,
1-\frac 1 \alpha\right)$ be given. 
Then there exists $c>0$ such that 
$$\int_{n^{-1+\frac{1}{\alpha}+\varepsilon}}^{\frac \pi d}
{\mathbb E}\left[\prod_y \vert\varphi_\xi(tN_n(y))\vert{\bf
1}_{\Omega_n}\right] \, dt
=o(e^{-n^c}).$$
\end{prop}
To end the proof of Theorem \ref{thmTLL}, we observe that there exists 
$(\eta,\varepsilon)$ satisfying all the hypotheses of these propositions
(by taking $\eta>0$ small enough and
$\varepsilon<\frac{\alpha-1} {\alpha(3+2\beta(\alpha-1))}$ large enough).
\subsection{Proof of Proposition
\ref{lem:equivalent}.}
Remember that $V_n=\sum_{z\in\mathbb Z}N_n^\beta(z)$.
We start by a preliminary lemma.
\begin{lem}\label{lem:borne}
If $\beta>1$, then 
$$\sup_n {\mathbb E}\left[ \pare{\frac{n^{\delta}}{V_n^{\frac{1}{\beta}}}}
^{\frac \beta  {\beta-1}}\right]<+\infty.$$
If $\beta\le  1$, then for all $p\ge 1$, 
$$\sup_n  
{\mathbb E}\left[\left(\frac{n^\delta}{V_n^{\frac 1 \beta}}\right)^p\right]<+\infty.$$
\end{lem}
\noindent A direct consequence of this lemma is that the sequence $(n^\delta V_n^{-\frac 1 \beta},n\ge 1)$ 
is uniformly integrable.
\begin{proof}
We start with the case $\beta>1$. We already observed in the proof of Lemma \ref{lem:omega_n} that for every $n\geq 1$,
$$n\le V_n^{\frac 1 \beta} R_n^{1-\frac 1 \beta}.$$
But it is proved in \cite{LGR} Equation (7.a) that $\EE[R_n]=\O( n^{\frac 1 \alpha})$. The result follows. \\*
We suppose now that $\beta \le 1$. Since we have
\begin{equation}
\label{BS}  n = \sum_x N_n(x) \leq V_n (N_n^{*})^{1-\beta} \, ,
\end{equation} 
we get 
\begin{equation}
\label{BSa}
\frac{n^{\delta}}{V_n^{1/\beta}} 
\leq \pare{\frac{N_n^*}{n^{1-\frac{1}{\alpha}}}}^{\frac{1}{\beta}-1} \, .
\end{equation} 
We use next the fact that $N_n^*$ is a subadditive functional:
\begin{equation}
\label{subbaddN*}
N^*_{n+m} \leq N_n^* + N_m^* \circ \theta_n \, ,
\end{equation} 
where 
$$N_m^* \circ \theta_n :=\sup_x \sum_{k=0}^{m-1} {\bf 1}_{\{S_{n+k}=x\}}
= \sup_x \sum_{k=0}^{m-1} {\bf 1}_{\{S_{n+k}-S_n=x\}}\, ,$$
is independent of 
$\sigma(S_0,\cdots, S_{n-1})$. Moreover, $0 \leq N^*_{n+1} - N^*_n \leq 1$. 
Therefore, we can prove in exactly the
same way as for the range (see \eqref{SousaddRange} in the appendix), that 
\begin{equation} 
\label{subadd}
 {\mathbb P} \pare{ N_n^* \geq a +b } \leq {\mathbb P} \pare{ N_n^* \geq a } 
{\mathbb P}\pare{ N_n^* \geq b} \quad \text{for all } a,\, b \in \NN \, .
\end{equation}  
Now it is known (see for example \cite{Borodin3}) that $N_n^*/ n^{1-1/\alpha}$ converges
in distribution toward $\sup_x L_1(x)$.
Let $t > 0$, be such that $\PP\cro{\sup_x L_1(x) \geq t} \leq 1/2$.
Since 
$$\lim_{n \rightarrow \infty} {\mathbb P} \pare{N_n^* \geq \floor{t n^{1-1/\alpha}}}
\leq {\mathbb P}\pare{\sup_x L_1(x) \geq t} \leq 1/2,$$ 
we obtain that for $n$ large enough, ${\mathbb P} \pare{N_n^* \geq \floor{t n^{1-1/\alpha}}} 
\leq 2/3$. 
Hence for $n$ large enough, and all $ p\ge 1$, 
\begin{eqnarray} 
\nonumber
 {\mathbb E} \cro{\pare{\frac{N_n^*}{n^{1-1/\alpha}}}^p} 
& = & p \int_0^{\infty} x^{p-1}  {\mathbb P} \pare{N_n^* \geq x n^{1-1/\alpha}} \, dx
\leq p t^p \int_0^{\infty} u^{p-1}
{\mathbb P} \pare{N_n^* \geq t n^{1-1/\alpha} u} \, du
\\ 
\label{UI}
& \leq & p t^p \int_0^{\infty} u^{p-1}
{\mathbb P} \pare{N_n^* \geq \floor{t n^{1-1/\alpha}}}^{\floor{u}} \, du
\leq  p t^p  \int_0^{\infty} u^{p-1} 
\pare{\frac{2}{3}}^{\floor{u}} \, du \, ,
\end{eqnarray} 
where the first inequality in \refeq{UI} comes from \refeq{subadd}. 
Thus, for all $p \ge 1$, 
\begin{equation}
\label{contsup}
 \sup_n {\mathbb E} \cro{\pare{\frac{N_n^*}{n^{1-1/\alpha}}}^p} < \infty \, .
\end{equation} 
The lemma now follows from \refeq{BSa}.
\end{proof}

\noindent The next step is the
\begin{lem}\label{lem:gaussian}
Under the hypotheses of Proposition \ref{lem:equivalent}, we have
\begin{eqnarray*}
 \int_{|t| \le  n^{-\delta+ \overline{\eta}}} 
e^{-it\floor{n^{\delta }x}} \EE \cro{\left\{\prod_y \varphi_\xi(tN_n(y))
-e^{-|t|^\beta V_n(A_1+iA_2 \textrm{sgn}(t)) }
\right\}
   {\bf 1}_{\Omega_n}} \, dt
= o(n^{- \delta}) \, ,
\end{eqnarray*} 
uniformly in $x\in{\mathbb R}$, where $A_1$ and $A_2$ are the constants appearing in \eqref{phi0}. 
\end{lem}
\begin{proof}
It suffices to prove that 
\[
 \int_{|t| \le n^{-\delta+\overline{\eta}}}
 \va{\EE \cro{\prod_y \varphi_\xi(tN_n(y))  {\bf 1}_{\Omega_n}} 
 - \EE \cro{ e^{-|t|^\beta V_n(A_1+iA_2 \textrm{sgn}(t)) }{\bf 1}_{\Omega_n}}}
 \, dt = o(n^{- \delta}) \, .
 \] 
Set
$$E_n(t):=\prod_y\varphi_\xi(tN_n(y))-
\prod_y \exp\left(-|t|^\beta N_n^\beta (y)(A_1+iA_2 \textrm{sgn}(t))\right). $$
Observe that 
\begin{eqnarray*}
E_n(t) = \sum_y & & \left(\prod_{z<y}\varphi_\xi(tN_n(z))\right)
\left(\varphi_\xi(tN_n(y))- e^{-|t|^\beta N_n^\beta (y) (A_1+iA_2 \textrm{sgn}(t)) }\right)\\ 
         & \times & \left(\prod_{z>y}e^{-|t|^\beta N_n^\beta (z) (A_1+iA_2 \textrm{sgn}(t)) }\right).
\end{eqnarray*} 
But on $\Omega_n$, if $|t|\le n^{-\delta+\overline{\eta}}$, then 
\begin{eqnarray}
\label{tnnz}
|t| N_n(z)  \le n^{\eta+\overline{\eta}-\frac 1 {\alpha\beta}}.
\end{eqnarray} 
This implies in particular that $|t| N_n(z)< \varepsilon_0$ for $n$ large enough, since the hypothesis on $\eta$ implies $\eta+\overline{\eta}< 1/(\alpha\beta)$. 
Thus by using \eqref{majorationphi} we get
\begin{eqnarray*} 
|E_n(t)| \le  \sum_y\left\vert \varphi_\xi(tN_n(y))-
   \exp\left(-|t|^\beta N_n^\beta (y) (A_1+iA_2 \textrm{sgn}(t))\right)\right\vert
\exp\left(-\sigma |t|^\beta\sum_{z\ne y}N_n^\beta(z) \right),
\end{eqnarray*}
for $n$ large enough.  
Observe next that \eqref{phi0} implies 
$$\left\vert \varphi_\xi(u)-\exp\left(-|u|^\beta(A_1+iA_2\textrm{sgn}(u)\right)\right\vert\le
|u|^\beta h(\vert u\vert) \quad \textrm{for all } u\in \RR,$$ 
with $h$ a continuous and monotone function on $[0,+\infty)$ vanishing in $0$. 
Therefore by using \eqref{tnnz} we get
\begin{eqnarray*}
|E_n(t)|\le   |t|^\beta h(n^{\eta+\overline{\eta}-\frac{1}{\alpha\beta}})  \sum_y N_n^\beta(y) \exp\left(-\sigma|t|^\beta\sum_{z\ne
y}N_n^\beta(z)\right).
\end{eqnarray*}
Now on $\Omega_n$, according to \refeq{minVn} and the hypothesis on $\eta$, if $n$ is large enough, 
$$ \sum_{z\ne y}N_n^\beta(z)
\ge V_n/2\qquad \text{for all }y\in \ZZ.$$
By using this and the change of variables
$v=tV_n^{1/\beta}$, we get 
$$
\int_{|t| \le n^{-\delta+\overline{\eta}}}\mathbb{E}\left[\vert E_n(t)\vert
{\bf 1}_{\Omega_n}\right]\, dt\leq h(n^{\eta+\overline{\eta}-\frac 1 {\alpha\beta}})
\mathbb{E}[V_n^{-1/\beta}]
\int_{\mathbb{R}} \vert v\vert^\beta \exp\left(-\sigma \vert v\vert^\beta/2\right)\, dv
=o(\mathbb{E}[V_n^{-1/\beta}]),
$$
which proves the result according to Lemma \ref{lem:borne}.
\end{proof}
 
\noindent Finally Proposition \ref{lem:equivalent} follows from the

\begin{lem}\label{le13} Under the hypotheses
of Proposition \ref{lem:equivalent}, we have
\[ \frac{d}{2 \pi} \int_{|t| \le n^{-\delta + \overline{\eta}}} e^{-it \floor{n^\delta x}}
\EE \cro{ e^{-|t|^\beta V_n(A_1+iA_2 \textrm{sgn}(t)) }{\bf 1}_{\Omega_n}} \, dt
=d \frac{C(x)}{n^\delta}+ o(n^{-\delta}) \, ,
\]
uniformly in $x\in\mathbb R$. 
\end{lem}
\begin{proof}
Set 
$$I_{n,x}:=\int_{|t| \le n^{-\delta + \overline{\eta}}} e^{-it \floor{n^\delta x}}
 e^{-|t|^\beta V_n(A_1+iA_2 \textrm{sgn}(t)) } \, dt.$$
Since $|\floor{n^\delta x}-n^\delta x|\le 1$, for all $n$ and $x$, it is immediate that 
$$I_{n,x}=\int_{|t| \le n^{-\delta + \overline{\eta}}} e^{-it n^\delta x}
 e^{-|t|^\beta V_n(A_1+iA_2 \textrm{sgn}(t)) } \, dt + \O(n^{-2\delta+2\overline{\eta}}).$$
But $2\overline{\eta}<1/(\alpha\beta) <\delta$ by hypothesis. So actually
$$I_{n,x}=\int_{|t| \le n^{-\delta + \overline{\eta}}} e^{-it n^\delta x}
 e^{-|t|^\beta V_n(A_1+iA_2 \textrm{sgn}(t)) } \, dt + o(n^{-\delta}).$$
Next, after some changes of variables, we get: 
\begin{eqnarray}
\label{Jn}
 \int_{|t| \le n^{-\delta + \overline{\eta}}} e^{-it n^\delta x}
 e^{-|t|^\beta V_n(A_1+iA_2 \textrm{sgn}(t)) } \, dt=n^{-\delta}\left\{2\pi\frac{n^\delta}
{V_n^{1/\beta}}f_\beta\left(\frac{{n^\delta x}}{V_n^{1/\beta}}\right)
-J_{n,x}\right\},
\end{eqnarray}
where 
$$
J_{n,x}:=  \int_{|v|\ge n^{\overline{\eta}}}\!\! 
e^{-iv x}
  e^{-|v|^\beta\frac{V_n}{n^{\beta\delta}}(A_1+iA_2 \textrm{sgn}(v)) } \, dv.
$$
Now it is known that $W_n:=n^{\delta}V_n^{-1/\beta}$ 
converges in distribution, as $n\to \infty$, toward 
$W:=|L|^{-1}_\beta$ (see \cite{ChenLiRosen} Lemma 14 or \cite{KestenSpitzer} Lemma 6). 
Then by Skorohod's representation Theorem, we can find a sequence $(\widetilde{W}_n,n\geq 1)$ and $\widetilde{W}$ 
distributed respectively as $(W_n,n\geq 1)$ and $W$ such that $\widetilde{W}_n$ 
converges almost surely toward $\widetilde{W}$. Moreover, Lemma \ref{lem:borne} 
ensures that the sequence $(\widetilde{W}_n,n\ge 1)$ is uniformly integrable, so actually the convergence holds in ${\mathbb L}^1$. 
Let us deduce that 
\begin{eqnarray}
\label{WnW} 
{\mathbb E}[g_x(W_n)]={\mathbb E}[g_x(W)]
+o(1),
\end{eqnarray} 
where $g_x:z\mapsto zf_\beta(xz)$ and the $o(1)$ is uniform in $x$. First 
\begin{eqnarray*}
\left\vert{\mathbb E}[g_x(W_n)]-{\mathbb E}[g_x(W)]\right\vert
  &\le &   \sup_{x,z\in\mathbb R}\vert (g_x)'(z) \vert{\mathbb E}[\vert \widetilde{W}_n-\widetilde{W}\vert]\\
  &\le & \sup_u \vert f_\beta(u)+u f_\beta'(u)  \vert{\mathbb E}[\vert \widetilde{W}_n-\widetilde{W}\vert].
\end{eqnarray*}
But remember that 
$$f_\beta(u)=\frac 1{2\pi}\int_\RR e^{itu}e^{-\vert t\vert^
\beta(A_1+iA_2 \textrm{sgn}(t))}\, dt\, .$$
So after differentiation under the integral sign and integration by parts we get
$$uf'_\beta(u)=-\frac 1{2\pi}\int_\RR e^{itu}(1-\beta \textrm{sgn}(t) 
\vert t\vert^{\beta}
(A_1+iA_2\textrm{sgn}(t)))e^{-\vert t\vert^
\beta (A_1+iA_2 \textrm{sgn}(t))}\, dt.$$
In particular $\sup_u \vert f_\beta(u)+u f_\beta'(u)  \vert$ is finite, and this proves \eqref{WnW}.

\noindent In view of \eqref{Jn} it only remains to prove that $\EE[J_{n,x}{\bf 1}_{\Omega_n}]=o(1)$. But this follows from the basic inequality
$$
\EE[|J_{n,x}{\bf 1}_{\Omega_n}|]\le \int_{|v|\ge n^{\overline{\eta}}}
  \EE\left[e^{-A_1|v|^\beta\frac{V_n}{n^{\beta\delta}} }{\bf 1}_{\Omega_n}\right] \, dv,
$$
and from the lower bound for $V_n$ given in \eqref{minVn}. 
\end{proof}

\subsection{Proof of Proposition
\ref{sec:step1}.}

Recall that on $\Omega_n$, $N_n(y) \le n^{1-\frac{1}{\alpha} +\eta}$, for all $y \in \ZZ$. 
Hence by \eqref{majorationphi},
\[
\int_{n^{-\delta+\overline{\eta}}}^{\varepsilon_0 n^{-1+\frac{1}{\alpha}-\eta}}
 {\mathbb E}\left[\prod_y \vert\varphi_\xi(tN_n(y))\vert{\bf 1}_{\Omega_n}\right]
\, dt
\leq \int_{n^{-\delta+\overline{\eta}}}^{\varepsilon_0 n^{-1+\frac{1}{\alpha}-\eta}} {\mathbb E} 
\cro{ \exp \pare{-\sigma t^\beta V_n} {\bf 1}_{\Omega_n}} \, dt \, .
\]
But on $\Omega_n$, we can also use the lower bound for $V_n$ given in \eqref{minVn}, which implies that
$$
\int_{n^{-\delta+\overline{\eta}}}^{\varepsilon_0 n^{-1+\frac{1}{\alpha}-\eta}}
 {\mathbb E}\left[\prod_y \vert\varphi_\xi(tN_n(y))\vert{\bf 1}_{\Omega_n}\right]
\, dt
 \leq   e^{-\sigma n ^{c\eta}},$$
for some constant $c>0$, depending on $\beta$. This proves the proposition.  
\subsection{Proof of Proposition
\ref{sec:step2}.}
First note that by using again \eqref{majorationphi} we get
\begin{eqnarray}
\label{majorationphin} 
\prod_y \vert\varphi_\xi(tN_n(y))\vert
\le \exp\left(-\sigma t^\beta \sum_{z:N_n(z)\le \varepsilon_0 
n^{1-\frac 1 \alpha -\varepsilon}}N_n^\beta(z)\right) \quad \text{for all }t\le n^{-1+\frac 1 \alpha + \varepsilon}.
\end{eqnarray}
The proof will then be a consequence of the  
\begin{lem}\label{sec:lem2b}
Under the hypotheses of Proposition \ref{sec:step2},  
for $n$ large enough and on $\Omega_n$, we have 
$$\# \left\{z\ :\ \frac{\varepsilon_0}{10}n^{1-\frac 1 \alpha - \varepsilon} \le N_n(z) \le \varepsilon_0 n^{1-\frac 1 \alpha - \varepsilon}\right\}
\ge \left(\frac{\varepsilon_0}{10}\right)^{\frac{2}{\alpha-1}} n^{\frac 1 \alpha -\frac{2\varepsilon+\gamma}{\alpha-1}}.$$
\end{lem}
\noindent Indeed according to this lemma and \eqref{majorationphin}, we get for $n$ large enough and on $\Omega_n$,
\begin{eqnarray*}
\prod_y \left\vert\varphi_\xi(tN_n(y))
\right\vert
&\le & \exp\left(-\sigma' n^{-\beta(1-\frac 1 \alpha +\eta)} n^{\frac 1 \alpha - \frac{2\varepsilon + \gamma}{\alpha-1}}n^{\beta(1 -\frac 1 \alpha -\varepsilon)}\right)\\
 & \le  &\exp\left(-\sigma' n^{\frac 1 \alpha - \beta (\eta+\varepsilon) -\frac{2\varepsilon + \gamma}{\alpha-1}}\right) \hspace{1cm} \text{for all }\varepsilon_0n^{-1+\frac 1 \alpha -\eta}\le t\le n^{-1+\frac 1 \alpha + \varepsilon},
\end{eqnarray*} 
for some constant $\sigma'>0$. This proves Proposition \ref{sec:step2}, since the hypothesis on $\varepsilon$ and $\gamma$ implies that  
$$\frac 1 \alpha - \beta(\eta+\varepsilon) -\frac{2\varepsilon + \gamma}{\alpha-1}> \frac 1 \alpha - 2\beta\varepsilon -\frac{3\varepsilon}{\alpha-1}>0.$$

\begin{proof}[Proof of Lemma \ref{sec:lem2b}]
Let $y_1$ be such that $N_n(y_1)=N_n^*=\sup_z N_n(z)$. Since $n=\sum_zN_n(z)\le N_n^*R_n$, we have $N_n(y_1)\ge n^{1-\frac 1 \alpha -\gamma}$, on $\Omega_n$.
Set 
$$y_0:=\min\left\{y \ge  y_1\ :\ N_n(y)\le \frac{\varepsilon_0}{2} n^{1 - \frac 1 \alpha-\varepsilon}\right\}.$$
Observe that  $y_0>y_1$ for $n$ large enough, since $\varepsilon > \gamma$ by hypothesis. In particular 
$$N_n(y_0-1)> \frac{\varepsilon_0}{2} n^{1 - \frac 1 \alpha-\varepsilon} \ge N_n(y_0).$$
But on $\Omega_n$,
\[N_n(y_0-1)-N_n(y_0)\le n^{(1-\frac 1 \alpha +\gamma)/2}
\, .
\]
Moreover, the hypotheses made on $\gamma$ and $\varepsilon$ imply that $\gamma<(1-1/\alpha)/3$ and $\varepsilon < (1-1/\alpha)/3$. Thus $\varepsilon<(1-1/ \alpha-\gamma)/2$, or equivalently $(1-1/\alpha + \gamma)/2 < 1- 1/\alpha - \varepsilon$. Therefore 
\begin{eqnarray}
\label{Nn0}
\frac{\varepsilon_0}{4} n^{1 - \frac 1 \alpha-\varepsilon} \le N_n(y_0) \le \frac{\varepsilon_0}{2} n^{1 - \frac 1 \alpha-\varepsilon},
\end{eqnarray}
for $n$ large enough. Next if $|y_0-z|\le \left(\frac{\varepsilon_0}{10}\right)^{\frac{2}{\alpha-1}} n^{\frac 1 \alpha -\frac{2\varepsilon + \gamma}{\alpha -1}}$, then on $\Omega_n$, 
$$|N_n(z) - N_n(y_0)| \le \sqrt{|y_0-z|^{\alpha-1} n^{1-\frac 1 \alpha +\gamma}} \le \frac{\varepsilon_0}{10} n^{1-\frac 1 \alpha -\varepsilon}.$$
Together with \eqref{Nn0}, this proves the lemma. 
\end{proof}  
\subsection{Proof of Proposition
\ref{sec:step3}.}
Let $M$ and $N$ be two positive integers such that
${\mathbb P}(X_1=N)>0$ and ${\mathbb P}(X_1=-M)>0$.
We denote by ${\mathcal C}^+$ the $(M+N)$-uple
$(N,...,N,-M,...,-M)$ in which $N$ is repeated $M$ times
and then $-M$ is repeated $N$ times.
We denote by ${\mathcal C}^-$ the "symmetric" $(M+N)$-uple
$(-M,...,-M,N,...,N)$ in which $-M$ is repeated $N$ times
and then $N$ is repeated $M$ times. Set $T:=M+N$ and observe that
$$ p:={\mathbb P}((X_1,...,X_T)=\C^+)
  ={\mathbb P}((X_1,...,X_T)=\C^-)>0.$$

\noindent Let us notice that $(X_1,...,X_T)=\C^+$ corresponds to 
a trajectory going up to $MN$ (in $M$ steps) and then
coming back down to $0$ (in $N$ steps).
Analogously, $(X_1,...,X_T)=\C^-$ corresponds
to a trajectory that goes down to $-MN$ (in $N$ steps) and 
comes back up to $0$ (in $M$ steps).

\noindent We introduce now the event 
$$\D_n:=\left\{C_n  > \frac{np}{2T}\right\},$$
where
$$C_n:=\#\left\{k=0,...,\left\lfloor \frac n T\right\rfloor-1\ :\ 
(X_{kT+1},\dots,X_{(k+1)T})=\C^\pm\right\}.$$
Since the sequences $(X_{kT+1},\dots,X_{(k+1)T})$, for $k\ge 0$, are independent of each other, Chernoff's inequality implies that there exists $c>0$ such that
$${\mathbb P}(\D_n) =1-o(e^{-cn}).$$
\noindent We introduce now the notion of "peak".
We say that there is a peak based on $y$ at time $n$ if $S_n=y$ and 
$(X_{n+1},\dots,X_{n+T})= \C^\pm$.
We will see (in Lemma \ref{sec:p_n} below) that, on $\Omega_n \cap \D_n$, there is a large number of $y\in \ZZ$
on which are based a large number of peaks.
For any $y\in {\mathbb Z}$, let
$$
C_n(y):=\#\left\{k=0,\dots,\left\lfloor \frac n T\right\rfloor-1 \ :\ 
S_{kT}=y \textrm{ and }(X_{kT+1},\dots,X_{(k+1)T}) = \C^\pm\right\},
$$
be the number of peaks based on $y$ before time $n$ (and at times which are multiple 
of $T$), and let 
$$p_n :=\#\{y\in {\mathbb Z}\ :\
C_n(y)
\ge  n^{1-\frac{1}{\alpha}-2\gamma}\},$$
be the number of sites $y\in \ZZ$ on which at least $n^{1-\frac{1}{\alpha}-2\gamma}$ peaks are based.

\begin{lem}\label{sec:p_n}
On $\Omega_n \cap \D_n$, we have
$p_n\ge 3NM n^{\frac 1 \alpha- \alpha\gamma}$, for $n$ large enough. 
\end{lem}

\begin{proof}
Note that $C_n(y)\le N_n(y) $ for all $y\in \ZZ$. Thus on $\Omega_n \cap \D_n$, 
\begin{eqnarray*}
\frac{np}{2T} & \le & \sum_{ y \in {\mathbb Z}\ :\ C_n(y) < n^{1-\frac{1}{\alpha}-2\gamma}}
C_n(y) + \sum_{ y \in {\mathbb Z}\ :\ C_n(y) \ge n^{1-\frac{1}{\alpha}-2\gamma}}
C_n(y) 
\\
& \le &
n^{1-\frac{1}{\alpha}-2\gamma} R_n + N_n^* p_n
\\
&\le &  n^{1-\gamma} +p_n n^{1-\frac{1}{\alpha}+\frac{\alpha\gamma}2},
\end{eqnarray*}
according to (\ref{sec:majoN_n*}).
This proves the lemma. 
\end{proof}
We have proved that, if $n$ is large enough, the event $\Omega_n\cap\D_n$
is contained in the event 
$$\E_n:= \{p_n\ge 3NM n^{\frac 1 \alpha- \alpha\gamma}\}.$$
Now, on  $\E_n$, we define
$Y_i$  for $i=1,\dots,\left\lfloor{n^{\frac 1 \alpha-\alpha\gamma}}
   \right\rfloor$, by 
$$Y_1:=\min\left\{y\in {\mathbb Z}\ :\
C_n(y)\ge n^{1-\frac 1 \alpha -2 \gamma} \right\},$$
and
$$Y_{i+1}:= \min\left\{y\ge Y_i+3NM\ :\ C_n(y)\ge n^{1-\frac 1 \alpha -2 \gamma} \right\}\quad \textrm{for }i\ge 1 .$$
The $Y_i$'s are sites on which at least
$n^{1-\frac 1 \alpha -2 \gamma}$ peaks are based and are such 
that $\vert Y_i-Y_j\vert\ge 3NM$, if $i\ne j$.
For every $i=1,\dots,\left\lfloor{n^{\frac 1 \alpha-\alpha\gamma}}\right\rfloor$, 
let
$t_i^{1},\dots,t_i^{\left\lfloor n^{1-\frac 1 \alpha -2 \gamma}\right\rfloor}$ be the 
$\left\lfloor n^{1-\frac 1 \alpha -2 \gamma}\right\rfloor$
first times (which are multiples of $T$) when a peak is based on the site $Y_i$. 
We also define $N_n^{0}(Y_i+NM)$ as the
 number of visits of $S$ before time $n$ to $Y_i+NM$, which do not occur during the time intervals $[t_i^j,t_i^j+T]$, 
for $j\le \left\lfloor n^{1-\frac 1 \alpha -2 \gamma}\right\rfloor$.   

\begin{lem} 
\label{independance}
Conditionally to the event $\E_n$, 
$((N_n(Y_i+MN)-N_n^{0}(Y_i+MN),i\ge 1)$ is a sequence of independent
identically distributed random variables with binomial distribution
${\mathcal B}\left(\left\lfloor n^{1-\frac 1 \alpha -2 \gamma}\right\rfloor;
\frac 12\right)$.
Moreover this sequence is independent of $((N_n^0(Y_i+MN),i\ge 1)$. 
\end{lem} 
\begin{proof}
On $\E_n$, we have
$$N_n(Y_i+MN)-N_n^{0}(Y_i+MN)=\sum_{j=1}^{\lfloor n^{1-\frac 1\alpha-2\gamma}
  \rfloor}
   {\bf 1}_{\{(X_{t_i^j+1},...,X_{t_i^j+T})\in\C^+\}},$$ 
since the peaks based on the other $Y_k$'s cannot pass through $Y_i+MN$.
But conditionally to $\E_n$, the sequence 
$\left({\bf 1}_{\{(X_{t_i^j+1},...,X_{t_i^j+T})\in{\C^+}\}}\right)_{i,j}$
is a sequence of independent Bernoulli random variables with parameter $1/2$, which is independent of $(X_k, k\not\in
  \bigcup_{i,j}[t_i^j,...,t_i^j+T])$. Since $N_n^0(Y_i+MN)$ only depends on the values of the $X_k$'s for $k\not\in
  \bigcup_{i,j}[t_i^j,...,t_i^j+T]$, the result follows.
\end{proof} 

\noindent Let now 
$\rho:=\sup\{\vert\varphi_\xi(u)\vert\ :\ d\left(u,\frac{2\pi}d{\mathbb Z}\right)
\ge\varepsilon_0\}$. According to Formula 
\eqref{majorationphi},
\begin{eqnarray*} 
\va{\varphi_\xi(u)} 
& \leq & 
\rho {\bf 1}_{\{d\left(u,\frac{2\pi}d{\mathbb Z}\right) \geq \epsilon_0\}}
+ \exp\left(-\sigma d\left(u,\frac{2\pi}d{\mathbb Z}\right)^\beta \right)
{\bf 1}_{\{d\left(u,\frac{2\pi}d{\mathbb Z}\right) < \epsilon_0\}} 
\\
& \leq & \exp \pare{-\sigma n^{-\frac 1 \alpha + 2\alpha\gamma}}
\, , 
\end{eqnarray*} 
as soon as $d\left(u,\frac{2\pi}d{\mathbb Z}\right) \geq n^{-\frac{1}{\alpha \beta}+
\frac{2\alpha\gamma}{\beta}}$, and  
$\rho \leq \exp \pare{-\sigma n^{-\frac 1 \alpha + 2\alpha\gamma}}$. But recall 
that $\rho<1$ and $2\alpha^2\gamma < 1$. Therefore, for $n$ large enough,
\begin{eqnarray}
\label{majodist}
\prod_z \va{\varphi_{\xi}(t N_n(z)) }
\leq \exp \pare{- \sigma n^{-\frac 1 \alpha + 2\alpha\gamma}\# \acc{z\ :\ d\left(t N_n(z),\frac{2\pi}d{\mathbb Z}\right) \geq 
n^{-\frac{1}{\alpha \beta}+\frac{2\alpha\gamma}{\beta}}}} \, .
\end{eqnarray}
\noindent Then notice that
\begin{eqnarray}
\label{Gt}
d\left(t N_n(z),\frac{2\pi{\mathbb Z}}d\right)\ge
n^{-\frac{1}{\alpha \beta}+\frac{2\alpha\gamma}{\beta}} \Longleftrightarrow N_n(z)\in \I:=\bigcup_{k\in{\mathbb Z}}I_k,
\end{eqnarray}
where for all $k\in \ZZ$,
$$I_k:=\left[\frac{2k\pi}{dt}+\frac{n^{-\frac{1}{\alpha \beta}+\frac{2\alpha\gamma}{\beta}}}
{t},\frac{2(k+1)\pi}{d t}-\frac{n^{-\frac{1}{\alpha \beta}+\frac{2\alpha\gamma}{\beta}}}
{t}\right]. $$
In particular ${\mathbb R}\setminus\I=\bigcup_{k\in{\mathbb Z}}J_k$, where for all $k\in \ZZ$,
$$J_k:=\left(\frac{2k\pi}{dt}-\frac{n^{-\frac{1}{\alpha \beta}+\frac{2\alpha\gamma}{\beta}}}
{t},\frac{2 k\pi}{d t}+\frac{n^{-\frac{1}{\alpha \beta}+\frac{2\alpha\gamma}{\beta}}}
{t}\right). $$

\begin{lem}\label{sec:liminf}
Under the hypotheses of Proposition \ref{sec:step3},
for every $i\le \left\lfloor{n^{\frac 1 \alpha-\alpha\gamma}} \right\rfloor$, $t\in (n^{-1+\frac 1 \alpha+\varepsilon},\pi/d)$ and $n$ large enough,
$${\mathbb P}\left(N_n(Y_i+MN)\in \I\mid \E_n 
,\ N_n^0(Y_i+MN)\right) \ge \frac 13\quad \textrm{almost surely}.$$
\end{lem}

\noindent 
Assume for a moment that this lemma holds true 
and let us finish now the proof of Proposition \ref{sec:step3}. Lemmas 
\ref{independance} and \ref{sec:liminf} ensure that conditionally to 
$\E_n$ and $((N_n^0(Y_i+MN),i\ge 1)$, the events 
$\{N_n(Y_i+MN)\in\I\}$, $i\ge 1$, are independent of each other, and all happen
with probability at least $1/3$.
Therefore, since $\Omega_n\cap\D_n\subseteq \E_n$,
there exists $c>0$, such that
$${\mathbb P}\left(\Omega_n \cap \D_n,
\ \#\{i\ :\ N_n(Y_i+MN)\in{\mathcal I}\}\le
\frac{n^{\frac 1 \alpha-\alpha\gamma}} 4\right)
\le {\mathbb P}\left(B_n\le \frac{n^{\frac 1 \alpha-\alpha\gamma}}4\right) 
=o(\exp(-cn)),
$$
where for all $n\ge 1$, $B_n$ has binomial distribution
${\mathcal B}\left(\left\lfloor {n^{\frac 1 \alpha -\alpha\gamma}}\right\rfloor;
\frac 13\right)$.

\noindent But if $\#\{z\ :\ N_n(z)\in\I\}\ge
n^{\frac 1 \alpha-\alpha\gamma}/4 $,
then by \eqref{majodist} and \eqref{Gt} there exists a constant $c>0$, such that
$$\prod_z\vert\varphi_\xi(tN_n(z))\vert
\le \exp\left(-c n^{\frac 1 \alpha-\alpha\gamma}n^{-\frac 1 \alpha+2\alpha\gamma}\right),
$$
which proves Proposition \ref{sec:step3}.

\begin{proof}[Proof of Lemma \ref{sec:liminf}]
First notice that by Lemma \ref{independance}, for any $H\ge 0$,
\begin{eqnarray}
\label{H}
{\mathbb P}(N_n(Y_i+MN)\in\I \mid \E_n,\ N_n^0(Y_i+MN)=H)={\mathbb P}\left(H+b_n\in\I\right),
\end{eqnarray}
where $b_n$ is a random variable with
binomial distribution
${\mathcal B}\left(\left\lfloor {n^{1-\frac 1 \alpha -2\gamma}}\right\rfloor; \frac 12\right)$.
We will use the following result whose proof is postponed.

\begin{lem}\label{sec:lem0}
Under the hypotheses of Proposition \ref{sec:step3}, for
every $t\in\ (n^{-1+\frac 1 \alpha +\varepsilon},\pi/d)$ and 
for $n$ large enough, the following holds:
\begin{itemize}
\item[(i)] For any integer $k$ such that all the elements of
$I_k-H$ are smaller than $\frac 12
\left\lfloor {n^{1-\frac 1 \alpha -2\gamma}} \right\rfloor$, 
$${\mathbb P}(b_n\in (I_k-H))\ge
   {\mathbb P}(b_n\in (J_k-H)) . $$
\item[(ii)] For any integer $k$ such that all the elements of
$I_{k}-H$ are larger than $\frac 12
\left\lfloor {n^{1-\frac 1 \alpha -2\gamma}} \right\rfloor$, 
$${\mathbb P}(b_n\in (I_k-H))\ge
  {\mathbb P}(b_n\in (J_{k+1}-H)) . $$
\end{itemize}
\end{lem}

\noindent Now call $k_0$ the largest integer satisfying
the condition appearing in (i) and $k_1$ the smallest integer satisfying
the condition appearing in (ii).
We have $k_1=k_0+1$ or $k_1=k_0+2$. According to Lemma \ref{sec:lem0},
we have
\begin{eqnarray*}
{\mathbb P}\left(H+b_n\in\I\right)
& \ge & \sum_{k \le k_0} \PP\pare{H+b_n\in I_k} + 
		\sum_{k \ge k_1} \PP\pare{H+b_n \in I_k}
\\
&\ge &  \sum_{k \le k_0} \PP\pare{H+b_n \in J_k} 
	+  \sum_{k \ge k_1} \PP\pare{H+b_n \in J_{k+1}} 
\\
& = &   {\mathbb P}(H+b_n\not\in\I)-{\mathbb P}
(H+b_n\in J_{k_0+1}\cup J_{k_1}). 
\end{eqnarray*}
Hence, 
\[ \PP \pare{H+b_n \in \I} 
\geq \frac 1 2 \cro{ 1 - \PP(H+b_n\in J_{k_0+1}\cup J_{k_1})} \, .
\] 
Let $\bar{b}_n := 2 \pare{b_n - \frac 1 2 \floor{n^{1-\frac 1 \alpha -2\gamma}}}
\floor{n^{1-\frac 1 \alpha -2\gamma}}^{-1/2}$, so that $\bar{b}_n$ converges
in distribution to a standard normal variable, whose distribution function
is denoted by $\Phi$. The interval $J_{k_1}$ being 
of length $2n^{-\frac{1}{\alpha \beta} + \frac{2\alpha\gamma}{\beta}}/t$, 
\begin{eqnarray*} 
\PP(H+b_n \in J_{k_1}) & = & \PP(\bar{b}_n \in [m_n,M_n])
\, , \mbox{ with } M_n-m_n = 4 \frac{n^{-\frac{1}{\alpha \beta} + \frac{2\alpha\gamma}{\beta}}}
{t \sqrt{\floor{n^{1-\frac 1 \alpha - 2\gamma}}}} 
\\
& \leq & \Phi(M_n) - \Phi(m_n) + 
\frac{C}{\sqrt{n^{1-\frac 1 \alpha - 2 \gamma}}}
\,  \mbox{ (by the Berry--Esseen inequality) }
\\
& \leq & \frac {M_n-m_n}{\sqrt{2\pi}} 
+ \frac{C}{\sqrt{n^{1-\frac 1 \alpha - 2 \gamma}}} 
\\
& \leq & C' n^{\frac 12 +\frac 1 {2\alpha}+\gamma+\frac{2\alpha\gamma}{\beta}-\frac{1}{\alpha\beta}-\frac 1 \alpha  - \varepsilon} 
+ \frac{C}{\sqrt{n^{1-\frac 1 \alpha - 2 \gamma}}} 
\, , 
\end{eqnarray*} 
for $t \ge n^{-1+\frac 1 \alpha+\varepsilon}$, and some constants $C>0$ and $C'>0$. 
Since $\alpha \le 2$, $\beta\le 2$, $\gamma < \frac 1 2 \frac{\alpha-1}{\alpha}$,
 and $\varepsilon>2\alpha\gamma/\beta + \gamma$ by hypothesis, 
we conclude that $\PP(H+b_n \in J_{k_1}) = o(1)$. The same holds for  
$\PP(H+b_n \in J_{k_0+1})$, so that for $n$ large enough, 
\[ \PP \pare{H+b_n \in \I} 
\geq \frac 1 2 \cro{ 1 - o(1)} \geq \frac 1 3 \, .
\]
Together with \eqref{H}, this concludes the proof of Lemma \ref{sec:liminf}. 
\end{proof}

\begin{proof}[Proof of Lemma \ref{sec:lem0}]
We only prove (i), since (ii) is similar.
So let $k$ be an integer such that all the elements of
$I_{k}-H$ are smaller than 
$\frac 12\left\lfloor {n^{1-\frac 1 \alpha -2 \gamma}}\right\rfloor$.
Assume that $(J_k-H)\cap \mathbb Z$ contains at least one nonnegative integer
(otherwise ${\mathbb P}(b_n\in (J_k-H))=0$  and there is nothing to prove).
Let $z_k$ denote the greatest integer in $J_k - H$, so that by our
assumption $\PP(b_n = z_k) > 0$ (remind that $0 \le z_k 
<\frac 12\left\lfloor {n^{1-\frac 1 \alpha - 2 \gamma}}\right\rfloor$). By monotonicity of the function
$z \mapsto  \PP(b_n=z)$, for $z\le \frac 12\left\lfloor n^{1-\frac 1 \alpha - 2 \gamma}\right\rfloor$, we get 
\[ 
\PP (b_n \in J_k - H) \leq \PP(b_n = z_k) \#((J_k -H)\cap \ZZ)
\leq \PP(b_n = z_k)  \ceil{\frac{ 2n^{-\frac 1 {\alpha\beta} + \frac{2\alpha\gamma}{\beta}}}{t}}
\, .
\]
In the same way,  
\[
\PP (b_n \in I_k - H) \geq  \PP(b_n = z_k) \#((I_k -H)\cap \ZZ)
\geq  \PP(b_n = z_k)  \floor{\frac {2\pi}{d t} -  
\frac{2n^{-\frac 1 {\alpha\beta} + \frac{2\alpha\gamma}{\beta}}}{t}}  
\, .
\]
Hence 
$$
\PP (b_n \in I_k - H) \geq \frac{ \floor{\frac{2\pi}{ d t}-\frac{2 n^{-\frac 1 {\alpha\beta} + 
\frac{2\alpha\gamma}{\beta}}}{t}}}
{\ceil{\frac{ 2n^{-\frac 1 {\alpha\beta} + \frac{2\alpha\gamma}{\beta}}}{t}}} 
\PP (b_n \in J_k - H)
\, .
$$
But $\pi/(dt)\ge 1$ and $2\alpha^2\gamma<1$ by hypothesis. It follows immediately that 
for $n$ large enough, we have $2n^{-\frac 1{\alpha\beta}+\frac{2\alpha\gamma}\beta}
<\pi/(2 d)$, 
and so
$$
\floor{\frac{2\pi}{ d t}-\frac{2 n^{-\frac 1 {\alpha\beta} + \frac{2\alpha\gamma}{\beta}}}{t}}
\ge \floor{\frac {3\pi}{2dt}}\ge 1+\floor{\frac \pi{2dt}} \ge \ceil{\frac\pi{2dt}}\ge
\ceil{\frac{ 2n^{-\frac 1 {\alpha\beta} + \frac{2\alpha\gamma}{\beta}}}{t}}\, .
$$
This concludes the proof of the lemma. 
\end{proof}
%
%
%
%
\section{Lattice case, $\alpha<1$: Proof of Theorem \ref{thmTLL2}} 

We only sketch the proof, since it is very similar and simpler than in the case 
$\alpha>1$. In particular we keep the same notation, for instance for $N_n^*$,
$R_n$, $V_n$, $\varepsilon_0$,...

\noindent We first introduce the analogue $\Omega'_n$ of $\Omega_n$: 
$$\Omega'_n=\Omega'_n(\varepsilon):= \left\{N_n^*\le n^\varepsilon\right\},$$
which is well defined for any $\varepsilon$. Note that on $\Omega'_n$, we have 
\begin{eqnarray}
\label{RnVn}
 V_n\ge R_n \ge n^{1-\varepsilon}.
\end{eqnarray}
Since $N_n^*=\sup_{k=0}^{n-1} \left[N_n(S_k)-N_k(S_k)\right]$, we obtain that 
\[
\PP \pare{N_n^{*} \geq n^{\varepsilon}} \leq n \PP \pare{\sup_mN_{m}(0)
\geq n^{\varepsilon}} \leq np_0^{n^\varepsilon-1},
\]
where $p_0:={\mathbb P}\pare{\exists k\ge 1\ : \ S_k=0}$. Since $\alpha<1$, the random walk $S$ is transient and $p_0<1$. It follows 
  that $\PP(\Omega'_n)=1-o(\exp(-n^{c}))$, for some constant $c>0$, and 
we can restrict our study to this set.  
Moreover, it is known (see for instance the introduction in \cite{KestenSpitzer} for an argument) that  
$$\frac 1 n V_n=\frac{1}{n}\sum_{y\in \ZZ} N_n^\beta(y) \mathop{\longrightarrow}
     _{n\rightarrow +\infty} \EE[\widetilde N_\infty^{\beta-1}(0)]=r^{-\beta
}\ \ \mbox{\rm a.s..}$$
We claim now that $\left(n^{1/\beta}V_n^{-1/\beta},n\geq 1\right)$
is uniformly integrable.
Indeed, if $\beta \ge 1$, this comes from the fact
that $V_n$ is larger than $n$, and
when $\beta<1$, this follows from the 
\begin{lem}
If $\beta<1$, there exists $\gamma >0$ such that 
\begin{equation} \label{UIa<1} 
\sup_n \EE \cro{\exp\pare{ \gamma \frac{n}{V_n} }}< \infty \, .
\end{equation} 
\end{lem}
\begin{proof}
Since $n = \sum_x N_n(x)$, H\"{o}lder's inequality gives
\[ \frac{n}{V_n} \leq \frac{\sum_x N_n(x)^2}{n} \, .
\]
Since  
$$\frac{1}{n} \sum_x N_n(x)^2=\frac{1}{n}\sum_{k=0}^{n-1} N_n(S_k)\, ,$$
Jensen's inequality gives
\[
\exp \pare{\gamma \frac{\sum_x N_n(x)^2}{n}}
\leq \frac{1}{n} \sum_{k=0}^{n-1} \exp \pare{\gamma N_n(S_k)} \, .
\]
Hence, 
\begin{eqnarray*}
\EE \cro{\exp \pare{\gamma \frac{\sum_x N_n(x)^2}{n}}}
& \leq  & \frac{1}{n} \sum_{k=0}^{n-1} \EE \cro{\exp \pare{\gamma N_n(S_k)}} 
\\
&\leq &  \EE \cro{\exp \pare{\gamma \widetilde N_{\infty}(0)}}
\, .
\end{eqnarray*}
Then, (\ref{UIa<1}) directly follows from the fact that 
$\widetilde N_{\infty}(0)$ is equal to 1 plus the sum of two independent
geometric variables with positive parameter, and thus has finite exponential moments.
\end{proof}
Let $\varepsilon>0$ and $\eta>0$ be such that $\eta +\varepsilon < 1/\beta$ 
and $\varepsilon < \eta \beta< 1/2$.
As in the proof of Proposition \ref{lem:equivalent}, we deduce that 
$$
\frac d {2\pi} \int_{|t| \le n^{-\frac{1}{\beta}+ \eta}}
e^{-it \floor{n^{\frac{1}{\beta}} x}}  \EE \cro{\prod_y \varphi_{\xi}(tN_n(y))} \, dt 
= \frac{D(x)}{n^{\frac 1 \beta}} + o(n^{- \frac 1 \beta}) \, ,
$$
where the $o(n^{-1/\beta})$ is uniform in $x$.
It remains to prove that 
\begin{equation}\label{sec:EQ0} 
\frac d{2\pi}\int_{ n^{-\frac 1\beta+\eta} }
    ^{\frac \pi {d}}
   \left\vert {\mathbb E}\left[
   \prod_y\varphi_\xi(tN_n(y)){\mathbf 1}_{\Omega'_n}\right]\right\vert\, dt=
   o(n^{-\frac 1\beta}).
\end{equation}
As in the proof of Proposition \ref{sec:step3} (see the beginning of Section 2.8. for the definitions of $\mathcal{D}_n$, $C_n, C_n(y),\ldots),$ we are led to prove that
$$\frac d{2\pi}\int_{ n^{-\frac 1\beta+\eta}}^{\frac \pi d}
   \left\vert {\mathbb E}\left[
   \prod_y\varphi_\xi(tN_n(y)){\mathbf 1}_{\Omega'_n\cap{\mathcal D}_n}
  \right]\right\vert\, dt=
   o(n^{-\frac 1\beta}).$$
Let $p'_n:=\#\{y\ :\ C_n(y)\ge 1\}$ be the random variable equal to the number of sites
of $\mathbb Z$ on which at least one peak is based.
Let us notice that on $\Omega'_n\cap\D_n$,
we have 
$$\frac{np}{2T}\le C_n=\sum_{y}C_n(y)\le\sum_{y\ :\ C_n(y)\ge 1} N_n(y)\le
       p'_n n^\varepsilon\, .$$ 
Thus $\Omega'_n\cap{\mathcal D}_n\subseteq{\mathcal E}'_n$, 
where ${\mathcal E}'_n:=\{p'_n\ge c_0 n^{1-\varepsilon}\}$, for a well chosen constant 
$c_0>0$.
As in the proof of Proposition \ref{sec:step3}, we construct $(Y_i)_i$
such that $C_n(Y_i)\ge 1$ and $Y_{i+1}-Y_i>MN$.
For every $i$, we define $N_n^0(Y_i+MN)$ as the number of visits to the site
$Y_i+MN$ without taking into account the possible visit during the first peak
based on $Y_i$.
Next we see that, on ${\mathcal E}'_n$, $(N_n(Y_i+MN)-N_n^0(Y_i+MN), i\le
c_0 n^{1-\varepsilon})$ is a sequence of i.i.d. random variables with
Bernoulli distribution with parameter $1/2$.

\noindent Let $t \in \left[ n^{-\frac 1\beta+\eta} 
,\frac\pi {d}\right]$.
We define the good and bad intervals respectively by
$$I'_k:=\left[\frac{2k\pi}{dt}+\frac {1}2,\frac{2(k+1)\pi}
    {dt}-\frac {1}2\right] 
  \ \mbox{and}\ J'_k:=\left(\frac{2k\pi}{dt}-\frac {1}2,\frac{2k\pi}
    {dt}+\frac {1}2\right).$$
Set also $\I':=\bigcup_{k\in\mathbb Z}I'_k$.
We observe that $J'_k$ is an open interval of length $1$
and $I'_k$ is a closed interval of length $2\pi/(dt)-1\ge 1$
(since $t\le \pi/d$).
Hence, if $N_n^0(Y_i+MN)$ is not in $\I'$, then
$N_n^0(Y_i+MN)+1$ is in $\I'$. This ensures that,
on $\E'_n$, $N_n(Y_i+MN)$ belongs to
$\I'$ with probability at least $1/2$.
Therefore, as after Lemma \ref{sec:liminf}, we get
$${\mathbb P}\left(\Omega_n\cap{\mathcal D}_n; \#\{i\ :\ N_n(Y_i+MN)\in\I'\}
     <\frac{c_0 n^{1-\varepsilon}}3\right) =o(n^{-\frac 1\beta}). $$
Hence, we just have to prove that
$$\frac d{2\pi}\int_{  n^{-\frac 1\beta+\eta}   }^{\frac \pi {d}}
   \left\vert {\mathbb E}\left[
   \prod_y\varphi_\xi(tN_n(y)){\mathbf 1}_{{\mathcal H}_{n,t}}
  \right]\right\vert\, dt=
   o(n^{-\frac 1\beta}),$$
with $    {\mathcal H}_{n,t}:=\left\{ \#\{y\ :\ N_n(y)\in\I'\}
     \ge \frac{c_0 n^{1-\varepsilon}}3 \right\}$. 
As after Lemma \ref{independance}, we notice that, if $n$ is large enough,
we have
$$d\left(u,\frac {2\pi}d{\mathbb Z}\right)\ge \frac{
 n^{-\frac 1\beta+\eta}  } 2
 \ \ \Rightarrow\ \   \vert\varphi_\xi(u)\vert\le \exp\left(
       -\frac{\sigma}{2^\beta}n^{-1+\beta\eta}\right) .$$
We notice also that if $N_n(y) \in \I'$, then 
 $d\left(tN_n(y),\frac {2\pi}d{\mathbb Z}\right)\ge t/2$, and thus
$d\left(tN_n(y),\frac{2\pi}d{\mathbb Z}\right)\ge 
  n^{-\frac 1\beta+\eta} /2$.
Now, on ${\mathcal H}_{n,t}$, we know that at least $c_0 n^{1-\varepsilon}/3$
sites $y$ satisfy this property. Therefore
$$
\left\vert {\mathbb E}\left[
   \prod_y\varphi_\xi(tN_n(y)){\mathbf 1}_{{\mathcal H}_{n,t}}
  \right]\right\vert\le
\exp\left( -\frac{c_0\sigma}{2^\beta 3}
   n^{1-\varepsilon}n^{-1+\beta\eta}\right)=o(n^{-\frac 1\beta}),
$$
since $\varepsilon<\beta\eta$. This gives \eqref{sec:EQ0} and achieves the proof of Theorem 2.
\hfill $\square$


\section{The strongly nonlattice case: Proof of Theorem \ref{nonlattice} }
We assume here that $\xi$ is strongly nonlattice. In that case, 
there exist $\varepsilon_0 >0$,
$\sigma > 0$ and $\rho <1$ such that 
\begin{itemize}
\item $|\varphi_{\xi}(u)| \leq \rho \hspace{2.4cm} \textrm{if } |u|\ge \varepsilon_0$,
\item $|\varphi_{\xi}(u)| \leq \exp(-\sigma |u|^{\beta}) \hspace{.5cm} \textrm{if } |u|<\varepsilon_0$.
\end{itemize}

\noindent 
{\bf Case $\alpha >1$.} We use here the notations of Section 2
with the hypotheses on $\gamma$, $\eta$, $\bar\eta$ and $\varepsilon$
of propositions \ref{lem:equivalent}, \ref{sec:step1}, \ref{sec:step2}
and \ref{sec:step3}.
Let $h_0$ be the density of Polya's distribution:
\[ h_0(y) = \frac1 \pi \frac{1-\cos(y)}{y^2} \, .
\]
Its Fourier transform is $\hat{h}_0(t) =(1 - |t|)_{+}$. For $\theta \in \mathbb{R}$,
let $h_{\theta}(y) = \exp(i \theta y) h_0(y)$ with Fourier transform
$\hat{h}_{\theta}(t)=\hat{h}_0(t+\theta)$. As is proved in \cite{durrett}
(see the proof of Theorem  5.4 p.114), it is enough to show
that for all $\theta \in \mathbb{R}$, 
\begin{equation}
\label{nonlattice.eq}
\lim_{n \rightarrow \infty} n^{\delta } \EE \cro{h_{\theta}(Z_n - n^\delta x)}
= C(x)\, \hat{h}_{\theta}(0) \, .
\end{equation} 
By Fourier inverse transform,
\[ 
n^{\delta } \EE \cro{h_{\theta}(Z_n - n^\delta x)}
= \frac{n^{\delta }}{2 \pi} 
 \int_{\mathbb R} e^{-iun^\delta x}
\EE \cro{\prod_{x \in \ZZ} \varphi_{\xi}(u N_n(x))} \hat{h}_{\theta}(u)\; du
\, .
\]
Since $\hat{h}_{\theta}\in L^1$, we can restrict our study to the 
event $\Omega_n$ of Lemma \ref{lem:omega_n}. 
The part of the integral corresponding to $|u| \leq n^{- \delta + \bar{\eta}}$
is treated exactly as in Proposition \ref{lem:equivalent}. The only change is that
we have to check that
\[
\lim_{n \rightarrow \infty} n^{\delta } \int_{|u| \le n^{-\delta + \bar{\eta}}}
\EE \cro{ e^{-|u|^{\beta} V_n (A_1 + i A_2 sgn(u))}{\bf 1}_{\Omega_n}} 
(\hat{h}_{\theta}(u) - \hat{h}_{\theta}(0)) \, du    = 0
\, ,
\]
which is obviously the case since  $2 \bar{\eta} < \delta$, 
using the fact that $\hat{h}_{\theta}$  is a Lipschitz function. 

Now since $\hat{h}_{\theta}$ is bounded, the part corresponding to 
$n^{-\delta + \bar{\eta}} \leq |u| 
\leq n^{-1 + \frac{1}{\alpha} + \varepsilon}$ doesn't need any
additional treatment. Actually, the proofs of Propositions \ref{sec:step1}  and \ref{sec:step2}  only use the behavior of $\varphi_{\xi}$ around $0$, which is
the same in the lattice or nonlattice case. 

We now turn our attention to the part of the integral corresponding 
to $|u| \geq n^{- 1 + \frac 1 \alpha  + \varepsilon}$ and prove that 
\begin{equation}
\label{nonlattice.pf} 
\lim_{n \rightarrow  \infty} n^{\delta} 
\int_{|u| \geq n^{- 1 + \frac 1 \alpha  + \varepsilon}}
e^{-iun^{\delta}x} \, \EE \cro{ \prod_x \varphi_{\xi}(u N_n(x)) 
{\bf 1}_{\Omega_n}}
\hat{h}_{\theta}(u) \, du = 0 \, .
\end{equation} 
To this end, note that 
\[
\va{ \EE \cro{ \prod_x \varphi_{\xi}(u N_n(x)) {\bf 1}_{\Omega_n}}}
\leq \EE \cro{\rho^{ \# \acc{x \, : \, \va{u N_n(x)} \geq \varepsilon_0}} 
{\bf 1}_{\Omega_n}} \, , 
\] 
and that  on $\Omega_n$, 
for $|u|  \geq  n^{- 1 + \frac 1 \alpha  + \varepsilon}$,
\begin{eqnarray*}
  n = \sum_x N_n(x) 
& \leq & \frac{\varepsilon_0}{|u|} R_n + 
N_n^*\,\,  \# \acc{x \, : \, \va{u N_n(x)} \geq \varepsilon_0} 
\\
& \leq & \varepsilon_0 n^{1-\varepsilon+\eta \beta/2} + n^{1-\frac 1 \alpha + \eta}
\,\,
 \# \acc{x \, : \, \va{u N_n(x)} \geq \varepsilon_0} 
\, .
\end{eqnarray*} 
Hence, since $\varepsilon > \eta \beta/2$, for $n$ large enough, on $\Omega_n$, 
and for  $|u|  \geq  n^{- 1 + \frac 1 \alpha  + \varepsilon}$,
\[
\# \acc{x \, : \, \va{u N_n(x)} \geq \varepsilon_0} 
\geq \frac 1 2 n^{\frac 1 \alpha - \eta}  \, .
\]
Therefore, for $n$ large enough,
\[
\va{ n^{\delta} 
\int_{|u| \geq n^{- 1 + \frac 1 \alpha  + \varepsilon}}
e^{-iun^{\delta}x} \EE \cro{ \prod_x \varphi_{\xi}(u N_n(x)) 
{\bf 1}_{\Omega_n}}
\hat{h}_{\theta}(u) \, du }
\leq n^{\delta} \rho^{\frac 1 2 n^{\frac 1 \alpha - \eta}} 
\int_{\mathbb R} \hat{h}_{\theta}(u) \, du \, , 
\]
which tends  to zero since $\eta < 1/ \alpha $.

\vspace{.3cm}
\noindent 
{\bf Case $\alpha <1$.} Using the notations and hypotheses on $\varepsilon,\eta,\gamma$ 
of Section 3, 
one has to prove that for all $\theta \in \mathbb{R}$ and 
all $x \in \mathbb{R}$,
\begin{equation}
\label{nonlatticetrans.eq}
\lim_{n \rightarrow \infty}
n^{1/\beta} \int_{\mathbb R} e^{-iun^{1/\beta}x} \, 
\EE \cro{\prod_x \varphi_{\xi}(u N_n(x))
{\bf 1}_{\Omega'_n}} \, \hat{h}_{\theta}(u) \, du = D(x) \hat{h}_{\theta}(0)
\, .
\end{equation} 
Again, the only change in the proof concerns the part of the integral
corresponding to $|u| \geq n^{-1/\beta + \eta}$. We use here the bound
\begin{eqnarray*} 
\va{\varphi_{\xi}(u N_n(x))} 
& \leq & \exp(- \sigma |u|^{\beta} N_n^{\beta}(x)) \,\, 
{\bf 1}_{\{\va{u N_n(x)} \leq \varepsilon_0\}} 
+ \rho  \,\, {\bf 1}_{\{\va{u N_n(x)} \geq \varepsilon_0\}}  
\\
& \leq & \exp(- \sigma n^{-1 +\eta \beta} N_n^{\beta}(x))
\,\, {\bf 1}_{\{\va{u N_n(x)} \leq \varepsilon_0\}} 
+ \rho \,\, {\bf 1}_{\{\va{u N_n(x)} \geq \varepsilon_0\}}  
\, .
\end{eqnarray*} 
If $\eta < 1/\beta$, then for $n$ large enough, $\rho \leq  
\exp(- \sigma n^{-1 +\eta \beta})$. Therefore, if $n$ is large enough, then for 
all $x$ and $u$ such that $N_n(x) \geq 1$ and $|u| \geq n^{-1/\beta + \eta}$, we have
\[
\va{\varphi_{\xi}(u N_n(x))} 
\leq \exp(- \sigma n^{-1 +\eta \beta} ) \, .
\]
Hence,
\[
\va{\EE\cro{\prod_x \varphi_{\xi}(u N_n(x)) \, {\bf 1}_{\Omega'_n}}}
\leq \EE\cro{\exp(- \sigma n^{-1 + \eta \beta} R_n )  {\bf 1}_{\Omega'_n}}
\leq \exp(- \sigma n^{\eta \beta - \varepsilon}) \, .
\] 
Therefore, since $\varepsilon < \eta \beta$,
\[
\lim_{ n \rightarrow \infty} 
n^{1/\beta} \int_{|u| \geq n^{-1/\beta + \eta}} 
e^{-iun^{1/\beta}x} \, \EE\cro{\prod_x \varphi_{\xi}(uN_n(x)) \, 
{\bf 1}_{\Omega'_n}} \hat{h}_{\theta}(u) \, du = 0 \, .
\]
This concludes the proof of Theorem \ref{nonlattice}. 
\hfill $\square$


\section{Random walks on randomly oriented lattices}
\subsection{Model and result}
We consider parallel moving pavements with different fixed speeds, independently chosen at the beginning with the same distribution. 
We study the random walk $(M_n, n\ge 0)$ representing the position of a walker
who at each time stays on the same moving pavement with probability $p \in (0,1)$, 
or jumps to another one with probability $1-p$.

Let us define $(M_n, n\ge 0)$ more precisely.
Let $\mu_X$ be a distribution on $\ZZ$ in the normal domain of attraction of a centered stable distribution with index $1<\alpha\le 2$ and density function denoted by $f_\alpha(\cdot)$. Let also  $\xi:=(\xi_y, y\in\ZZ)$ be a sequence of independent centered 
$\mathbb Z$-valued 
random variables with distribution $\mu_\xi$ belonging to the
normal domain of attraction of a stable distribution with index $1<\beta\leq 2$ 
and density function denoted by $f_{\beta}(\cdot)$. 
For each $y\in\mathbb Z$, $\xi_y$ will be the only horizontal
displacement allowed on line $y$.
Let $p\in (0,1)$. Given $\xi$, the random walk  $(M_n=(M_n^{(1)},M_n^{(2)}),n\ge 0)$  
is a Markov chain starting from $M_0:=(0,0)$
and such that at time $n+1$, it moves either horizontally
of $\xi_{M_n^{(2)}}$ (with probability
$p$) or makes a vertical jump with distribution $\mu_X$ (with probability $(1-p)$), i.e. 
$${\mathbb P}\left(M_{n+1}-M_n=(\xi_{M_n^{(2)}},0) \mid
      \xi,\ M_1,...,M_n\right) =p \ \ \mbox{if}\ \xi_{M_n^{(2)}}\ne 0,$$
$${\mathbb P}\left(M_{n+1}-M_n=(0,x) \mid
      \xi,\ M_1,...,M_n\right) =(1-p)\mu_X(x)\ \ \mbox{if}\ x\ne 0$$
and
$${\mathbb P}\left(M_{n+1}=M_n \mid
      \xi,\ M_1,...,M_n\right) =p+(1-p)\mu_X(0)\ \ \mbox{if}\ \xi_{M_n^{(2)}}= 0.$$
These random walks were first introduced by Campanino and P\'etritis in \cite{CP}
in the particular case when $p=1/3$ and when $\mu_X$ and $\mu_{\xi}$ are Rademacher 
distributions, i.e. take values $\pm 1$ with probability $1/2$. 
They proved the transience of 
$M$ as well as a law of large numbers. In \cite{GPLN2}, Guillotin-Plantard and Le Ny established 
the link between the Campanino and P\'etritis random walk and the random walk in random 
scenery and proved a functional limit theorem for the first one. It was also conjectured 
there that the probability of return to the origin of the Campanino and P\'etritis 
random walk is equivalent to a constant times $n^{-5/4}$. We prove this result here, as well as a generalization to the case of the random walks $M$ considered above.

\noindent To state our result, we will use 
the following representation of $M$:

\noindent Let $X:=(X_n,n\ge 1)$ be a sequence of independent random variables with distribution $\mu_X$. 
The random variable $X_n$ corresponds to the vertical move at time $n$ 
which will be chosen with probability $1-p$.
Let also $(\varepsilon_n,n\ge 0)$ be a sequence of independent Bernoulli random variables with parameter $p$, i.e. such that ${\mathbb P}(\varepsilon_1=1)=1-\PP(\varepsilon_1=0)=p$, and independent of $X$.
If $\varepsilon_n=1$, the particle $M$ moves horizontally at time $n$, otherwise
it moves vertically.
We then first define $S$ by $S_0:=0$ and  
$$S_n:=\sum_{k=1}^n  Y_k\quad \mbox{\rm for } n\geq 1\,  ,$$
where $Y_k:=X_k (1-\varepsilon_k)$. We next define $\widetilde{Z}$ by $\widetilde{Z}_0:=0$ and
$$\widetilde{Z}_n := \sum_{k=1}^n \xi_{S_{k-1}}\varepsilon_k=
    \sum_{y\in\mathbb Z} \xi_y \widetilde N_n(y)\quad \mbox{\rm for }  n\geq 1\, ,$$
where 
$$\widetilde N_n(y):=\#\{k=1,...,n\ :\ S_{k-1}=y\quad \textrm{and}\quad \varepsilon_k=1\}\, .$$
Then it is straightforward that the couple $(\widetilde{Z},S)$ has the same distribution as $M$.

\noindent We just notice that the process $S$ in this section is not exactly the same 
as in the previous sections (it is the same if we replace $X$ by $Y$).
However, we use the same notation just for convenience.

\noindent Now it is known that $(n^{-1/\alpha} S_{[nt]},t\ge 0)$ converges in distribution, 
when $n\to \infty$, to a L\'evy process $\tilde U=(\widetilde U_t,t\ge 0)$
where $\widetilde U=(1-p)^{\frac 1\alpha} U$ and 
$U$ is the process introduced in the introduction. 
We will use the fact that $(n^{-1/\alpha} S_{[nt]},t\ge 0\mid S_n=0)$ converges in distribution 
to $\widetilde U^0=(\widetilde U^0_t,t\ge 0)$ the associated bridge, i.e. 
the process $\widetilde U$ 
starting from $0$ and conditioned by $\widetilde U_1=0$. Let $(L_t^0(x),t\in [0,1], x\in{\mathbb R})$ be the local time process of $\widetilde U^0$ and set $\vert L^0\vert_{\beta}:=\left(\int_{{\mathbb R}} (L_1^0(x))^\beta\, dx\right)^{1/\beta}$.

\noindent Let $\varphi_\xi$ be the characteristic functions of $\xi_1$. Recall that $d$ is the positive integer 
such that $\{u\ :\ |\varphi_\xi(u)|=1\}=(2\pi/d){\mathbb{Z}}$. 
Let $d_0$ be the smallest positive integer $m$ such that
$\varphi_\xi(2\pi/d)^m=1$ and let $d_1$ be the greatest common divisor of
$\{m\ge 1\ :\ {\mathbb P}(X_1+...+X_m)>0\}$.

\begin{thm} \label{thrwrol} 
Assume that $d_1$ is a multiple of $d_0$, and let $E = d p^{-1}f_\alpha(0)f_\beta(0) \EE( |L^0|_\beta^{-1})$. Then, 
\begin{eqnarray*}
{\mathbb P}(M_n=(0,0))= \left\{ \begin{array}{ll} 
 E\times n^{-1-\frac 1{\alpha\beta}} + o(n^{-1-\frac 1{\alpha\beta}}) 
& \text{if $n$
 is a multiple of $d_0$;} \\
0 & \text{otherwise.}
\end{array}
\right. 
\end{eqnarray*}
\end{thm}

\begin{rqe} \emph{In the case of the Campanino and P\'etritis random walk, $d_0=d_1=2$. So the hypothesis of the theorem is satisfied.}
\end{rqe}

\begin{rqe} \emph{A corollary of our result is that the processes $M$ considered here are transient, this can be seen by using Borel--Cantelli lemma.}
\end{rqe}

\begin{rqe} \emph{It is most likely that an analogue result can be proved when $\alpha<1$ or $\beta\le 1$. We leave the details to the interested reader. In the same way one could certainly obtain similar estimates for the probabilities of return in $([n^\delta x],[n^{1/\alpha}y])$, with a constant $E$ depending on $x$ and $y$.}
\end{rqe}

\subsection{The event $\widetilde{\Omega}_n$}
Let $(N_n(y), y\in {\mathbb Z})$ and $R_n$ denote respectively the local time process and the range 
 of $S$ at time $n$:
\[ N_n(y) :=\#\{k=0,...,n-1\ :\ S_{k}=y \} \quad \mbox{\rm and}\quad  R_n := \# \acc{ y\ :\  N_n(y) > 0} \, . 
\]

\noindent For $\gamma>0$, set $\widetilde{\Omega}_n=\widetilde{\Omega}_n(\gamma):=\A_n\cap \B_n\cap \C_n$, where
$$\A_n:=\left\{ R_n\le n^{\frac 1\alpha+\gamma}\quad \textrm{and} \quad \sup_y N_n(y)\le 
n^{1-\frac 1\alpha+\gamma}\right\},$$
$$\B_n:=\left\{ \sum_{k=1}^n \varepsilon_k \ge \frac {np}2\right\}, $$
and
$$\C_n:= \left\{\sup_{y\neq z} \frac{\vert \widetilde{N}_n(y)-\widetilde{N}_n(z)\vert}{|y-z|^{\frac{\alpha-1}{2}}}
\le  n^{(1- \frac 1\alpha +\gamma)/2}\right\}. $$

\begin{lem}
\label{omegasec5}
For all $\gamma > 0$,  ${\mathbb P}(\widetilde{\Omega}_n)=1-o(n^{-1-\frac {1}{\alpha\beta}})$.
\end{lem}
\begin{proof}
According to the proof of Lemma \ref{lem:omega_n},  ${\mathbb P}(R_n\le n^{\frac{1}{\alpha}+\gamma})=1-o(n^{-1
-\frac 1{\alpha\beta}})$. Moreover, according to the proof of Lemma \ref{lem:borne} (see \eqref{contsup}), 
we have for all $\nu \geq 1$, 
\begin{eqnarray}
\label{Nnynu}
{\mathbb E}\left[\sup_y N_n^\nu(y)\right]= \O \left(n^{ \nu(1-\frac 1 \alpha)}\right).
\end{eqnarray}
Hence by the use of the Markov inequality, we get  
$${\mathbb P}\left(\sup_{y\in\mathbb Z} N_n(y)\ge 
n^{1-\frac{1}{\alpha}+\gamma}\right)=o(n^{-1
-\frac 1{\alpha\beta}}).$$
It follows that ${\mathbb P}(\A_n)=1-o(n^{-1
-\frac 1{\alpha\beta}})$.

\noindent Next it is well known that ${\mathbb P}(\B_n)=1-o(n^{-1-\frac 1{\alpha\beta}})$.

\noindent Finally, as in the proof of Lemma \ref{lem:omega_n}, the estimate of ${\mathbb P}(\C_n)$ comes from the following lemma:

\begin{lem}\label{Francoise}
For any integer $\nu\ge 1$, there exists a constant $C_\nu>0$ such that, for every
$n\ge 1$ and every $x,y\in{\mathbb Z}$
$$
{\mathbb E}\left[(\widetilde N_n(x)-\widetilde N_n(y))^{2\nu} \right]
  \le  C_\nu \vert x-y\vert^{\nu(\alpha-1)} n^{\nu(1-\frac 1 \alpha)} \, .
$$
\end{lem}
\begin{proof} 
Let $x$ and $y$ be fixed, with $x\neq y$ (otherwise, there is nothing
to prove). We have
\begin{eqnarray}
\label{NtildeN}
 \widetilde N_n(x) =p N_{n}(x) + \sum_{k=1}^n {\bf 1}_{\{S_{k-1}=x\}} \bar{\epsilon}_k \, ,
\end{eqnarray}
where  $\bar{\epsilon}_k= {\bf 1}_{\{\epsilon_k=1\}} -p$. 
Set $H_n(x) := \sum_{k=1}^n {\bf 1}_{\{S_{k-1}=x\}} \, \bar{\epsilon}_k$. For all 
$x \in \ZZ$, $(H_n(x), n \geq 1)$ is a martingale with respect to the 
filtration $\F_n = \sigma(X_k,\epsilon_k, k \leq n)$. Hence, $(H_n(x) -
H_n(y),n\ge 1)$ is a martingale as well. 
According to the B\"urkholder's inequality
(see \cite{HH} Theorem 2.11. p.23), for all integer $\nu \geq 1$, there exists a constant $C=C(\nu)$ 
such that for all $n \ge 1$,
\[
\EE \cro{(H_n(x)-H_n(y))^{2\nu}}^{\frac{1}{2\nu}} 
\leq C \acc{ 
\EE \cro{\pare{\sum_{k=1}^n \EE(d_k^2(x,y)\mid\F_{k-1})}^\nu}^{\frac{1}{2\nu}}
+ \EE\cro{\sup_{k=1,\dots,n} |d_k(x,y)|^{2 \nu}}^{\frac{1}{2\nu}}},
\]
where $d_k(x,y)$ is the martingale increment
\[ d_k(x,y) := H_k(x)-H_{k-1}(x) - H_k(y) + H_{k-1}(y)
	= \pare{{\bf 1}_{\{S_{k-1}=x\}} - {\bf 1}_{\{S_{k-1}=y\}}} \bar{\varepsilon}_k
\, .
\]
Note that for all $k \geq 1$, and all $x, y \in \ZZ$, $|d_k(x,y)| \leq 1$,
and that 
\[ \sum_{k=1}^n \EE(d_k^2(x,y)\mid \F_{k-1}) = \var(\varepsilon) 
\sum_{k=1}^n ({\bf 1}_{\{S_{k-1}=x\}} - {\bf 1}_{\{S_{k-1} =y\}})^2 
= \var(\varepsilon)  (N_{n}(y) + N_{n}(x)) \, .
\] 
Therefore,
\begin{eqnarray*} 
\EE \cro{(H_n(x)-H_n(y))^{2\nu}}^{\frac{1}{2\nu}} 
& \leq  & C \acc{ 1+ \EE\cro{N_{n}^{\nu}(y)}^{\frac{1}{2\nu}}
		 + \EE \cro{N_{n}^{\nu}(x)}^{\frac{1}{2\nu}}}\\
& \leq  & C (1 + 2n^{(1-1/\alpha)/2}) \quad \hbox{(by using \eqref{Nnynu})}\\
&\leq & 3 C n^{(1-1/\alpha)/2} |x-y|^{(\alpha-1)/2} 
\, ,
\end{eqnarray*} 
since $|x-y| \geq 1$, and $n \geq 1$. 
Hence, according to \cite{Jain-Pruitt} (see Equation \refeq{eq3}), 
\begin{eqnarray*} 
\EE\left\{(\widetilde N_n(x)-\widetilde N_n(y))^{2\nu}\right\}^{\frac 1 {2\nu}}
& \leq & p \EE\left\{(N_n(x)- N_n(y))^{2\nu}\right\}^{\frac 1 {2\nu}}
+ \EE\left\{(H_n(x)-H_n(y))^{2\nu}\right\}^{\frac 1 {2\nu}}
\\
& \leq & C_\nu n^{(1-1/\alpha)/2} |x-y|^{(\alpha-1)/2} \, ,
\end{eqnarray*} 
for some constant $C_\nu>0$. This proves Lemma \ref{Francoise}.  
\end{proof}
\noindent This concludes also the proof of Lemma \ref{omegasec5}. 
\end{proof}
\subsection{Expression of the return probability by an integral}
According to the result of the previous subsection, we are 
led to the study of ${\mathbb P}(\widetilde Z_n=0,S_n=0,\widetilde{\Omega}_n)$.
As in Lemma \ref{formule1}, we have~:
$$ \PP(M_n=(0,0),\widetilde{\Omega}_n)=\PP(\widetilde{Z}_n=0,S_n=0,\widetilde{\Omega}_n)
   =\frac 1{2\pi} \int_{-\pi}^{ \pi} {\mathbb E}
  \left[\prod_{y\in\mathbb Z}\varphi_\xi(t\widetilde N_n(y))
   {\bf 1}_{\{S_n=0\}}{\bf 1}_{\widetilde \Omega_n} \right]\, dt\, .$$
By following the proof of Lemma \ref{formule1} (note that a priori $\sum_y\widetilde N_n(y)$ is not equal to $n$ here), we get
\begin{eqnarray}
\label{intZtilde}
 {\mathbb P}(\widetilde Z_n=0,S_n=0,\widetilde \Omega_n)
   =\frac {d}{2\pi} \int_{-\pi/d }^{\pi/d} {\mathbb E}
  \left[\prod_{y\in\mathbb Z}\varphi_\xi(t\widetilde N_n(y))
  {\bf 1}_{\{\sum_y\widetilde N_n(y)\in d_0{\mathbb Z}\}} 
  {\bf 1}_{\{S_n=0\}}{\bf 1}_{\widetilde \Omega_n} \right]\, dt.
\end{eqnarray}

\noindent 
In the sequel we consider $\eta$, $\gamma$ and $\varepsilon$ satisfying all
the hypotheses of Section \ref{sec:scheme} and $\gamma< (\alpha-1)/(4\alpha)$.
\subsection{Estimate of the integral away from the origin}
The following is very similar to the case of RWRS.
\begin{lem} 
\label{resteint} 
We have 
$$\int_{n^{-\delta+\eta}}^{\pi/d}{\mathbb E}\left[
     \prod_{y\in\mathbb Z}\vert \varphi_\xi(t\widetilde N_n(y))\vert{\bf 1}_{\widetilde \Omega_n}\right]
    \, dt=o(n^{-1-\frac 1{\alpha\beta}}).$$
\end{lem} 
\begin{proof} First set 
$$\widetilde V_n:=\sum_{y\in\mathbb Z}\widetilde N_n(y)^\beta.$$
Since on $\widetilde \Omega_n$, $\sum_y  \widetilde N_n(y) = \sum_{k=1}^n \varepsilon_k \ge np/2$
and $\widetilde N_n(y) \le N_n(y) \le n^{1-\frac 1\alpha + \gamma}$, by following the proof of Lemma \ref{lem:omega_n}, we get on $\widetilde \Omega_n$: 
$$\widetilde V_n\ge c n^{\delta\beta-\gamma},$$
for some constant $c>0$. Let now $\varepsilon$ be as in Proposition \ref{sec:step2}. Then the proofs of Proposition \ref{sec:step1} and \ref{sec:step2} lead to 
$$\int_{n^{-\delta+\eta}}^{n^{-1+\frac 1{\alpha}+\varepsilon}}{\mathbb E}\left[
     \prod_{y\in\mathbb Z}\vert \varphi_\xi(t\widetilde N_n(y))\vert{\bf 1}_{\widetilde \Omega_n}\right]
    \, dt=o(n^{-1-\frac 1{\alpha\beta}}).$$
But we can also easily adapt the proof of Proposition \ref{sec:step3} to obtain~:
$$\int_{n^{-1+\frac 1{\alpha}+\varepsilon}}^{\pi/d }{\mathbb E}\left[
     \prod_{y\in\mathbb Z}\vert \varphi_\xi(t\widetilde N_n(y))\vert{\bf 1}_{\widetilde \Omega_n}\right]
    \, dt=o(n^{-1-\frac 1{\alpha\beta}}).$$
Indeed we just need to use "flat peaks" instead of peaks. These "flat peaks" are defined as follows. Let $M$ and $N$ be such that ${\mathbb P}(Y_1=N)>0$
and ${\mathbb P}(Y_1=-M)>0$.
Then an "upper flat peak" is a sequence of the type 
$$(Y_{H+1},\dots,
Y_{H+M},\varepsilon_{H+M+1}, Y_{H+M+2},\dots,Y_{H+M+N+1})=(N,\dots,N,1,-M,\dots,-M),$$ 
where $H$ is any multiple of $M+N+1$, and one can define analogously a "lower flat peak". 
We leave to the reader to check that we can then follow the proof of Proposition \ref{sec:step3} simply by replacing everywhere peaks by flat peaks. This concludes the proof of Lemma \ref{resteint}.  
\end{proof} 

\subsection{Estimate of the integral near the origin}
We turn now to the estimate of the integral in \eqref{intZtilde} on the interval $[-n^{-\delta+\eta},n^{-\delta+\eta}]$. For this we will roughly follow the same lines 
as for the proof of Proposition \ref{lem:equivalent}. However the technical details are more involved here, since we have to make all calculus conditionally to $\{S_n=0\}$. The first step is the following lemma: 

\begin{lem}
\label{lem:sec:normeVn} 
We have
\begin{equation}\label{sec:normeVn}
 \sup_n\ {\mathbb E} \left[ \left. \left( \frac{n^{\delta\beta}}
{\widetilde V_n} \right)^{\frac 1 {\beta-1} }
{\bf 1}_{\widetilde \Omega_n}\ \right\vert \ {S_n =0} \right]<+\infty.
\end{equation}
\end{lem}
\begin{proof}
Remind that on $\widetilde \Omega_n$, $np/2\le \sum_y\widetilde N_n(y)\le
  \widetilde V_n^{1/\beta} R_n^{1- 1/\beta}$. Observe on the other hand that $\delta\beta/(\beta-1)=\beta/(\beta-1)- 1/\alpha$. Thus there is a constant $C>0$ such that for all $n\ge 1$, on $\widetilde \Omega_n$, 
$$\left(\frac{n^{\delta\beta}}{\widetilde V_n}\right)^{\frac 1 {\beta-1}} \le C \frac{R_n}{n^{\frac 1 \alpha}}.$$ 
It follows from the above inequality that 
\[ 
 \EE\left[ \pare{\frac{n^{\delta\beta}}{\widetilde V_n}}^{\frac{1}{\beta - 1}} {\bf 1}_{\widetilde \Omega_n} {\bf 1}_{\{S_n=0\}}\right]
 \le  C
 \EE\left[ \frac{R_n}{n^{\frac 1 \alpha}} {\bf 1}_{\{S_n=0\}} \right].
\]
Set $m:=\lfloor n/2\rfloor$ and $m':=\lceil n/2\rceil$. By using that $R_n \leq R_{m'} +  \#\acc{ S_{m'+1}, \cdots, S_n} = R_{m'} +  \#\acc{ S_{m'+1}-S_n, \cdots, S_{n-1}-S_n, 0} $ and Markov property (respectively on the sequences $(S_k,k\geq 0)$ and $(S_n-S_{n-k},0\le k\leq n)$), we get 
\begin{eqnarray*} 
\EE\left[ \pare{\frac{n^{\delta\beta}}{\widetilde V_n}}^{\frac{1}{\beta - 1}} {\bf 1}_{\widetilde \Omega_n} {\bf 1}_{\{S_n=0\}}\right]
& \le  &  C \EE\left[\frac{ R_{m'}}{n^{\frac 1 \alpha}} \right] \times \sup_x \PP_x(S_m=0)
\\
&= &  \O\left( n^{-\frac 1 \alpha}\right)\,  ,
\end{eqnarray*}
since $\sup_x  \PP_x(S_m=0) = \O(n^{-1/\alpha})$ and   
$\EE(R_{m'}) = \O(n^{1/\alpha})$ (see \cite{LGR} Equation (7.a) p.703). We next divide all terms by $\PP(S_n=0)$ which is of order $n^{-1/\alpha}$ and this proves the lemma. 
\end{proof} 

\noindent We deduce the 

\begin{lem} \label{leminter} We have
$$\PP(\widetilde Z_n=0,S_n=0,\widetilde \Omega_n)=  n^{-1-\frac 1 {\alpha \beta}} d
 \, {\mathbb E}\left[  \frac{n^{\delta}}{\widetilde V_n^{\frac 1\beta}} 
    {\bf 1}_{ \{\sum_y\widetilde N_n(y)\in d_0{\mathbb Z}\} } 
   {\bf 1}_{\widetilde \Omega_n} \ \Big| \ S_n=0 \right]
      f_\alpha(0)f_\beta(0)  +o(n^{-1-\frac 1{\alpha\beta}}). $$
\end{lem} 
\begin{proof} By following the proof of Lemma \ref{lem:gaussian}, we see that,
uniformly on $\widetilde \Omega_n$, we have:
$$ \int_{|t|\le n^{-\delta + \eta}}
 \left\vert 
    \prod_y\varphi_\xi(t\widetilde N_n(y))- e^{-\vert t\vert^\beta
      \widetilde V_n(A_1+iA_2\textrm{sgn}(t))}\right\vert \, dt 
= o(\widetilde V_n^{-\frac 1 \beta})\, . $$
By using Lemma \ref{lem:sec:normeVn}, we get 
\begin{eqnarray*}
&&\int_{|t|\le n^{-\delta + \eta}}
 \EE \cro{ \va{
  		\prod_y\varphi_\xi(t\widetilde N_n(y))- e^{-\vert t\vert^\beta
     		 \widetilde V_n(A_1+iA_2\textrm{sgn}(t))} } {\bf 1}_{\widetilde \Omega_n}  {\bf 1}_{\{S_n=0\}}} \, dt\\
&= &  o(1)\times 
      \EE\cro{ \widetilde V_n^{-\frac 1 \beta} {\bf 1}_{\widetilde \Omega_n} {\bf 1}_{\{S_n=0\}}}\\
&= & o(n^{-\delta-\frac 1 \alpha})\times
\EE\cro{ \left. (n^{\delta\beta}\widetilde V_n^{-1})^{\frac 1 {\beta-1}}  {\bf 1}_{\widetilde \Omega_n} \ 
\right\vert\ S_n=0}^{\frac{\beta-1}{\beta}}    \quad (\mbox{since } \PP(S_n=0) = \O(n^{-\frac 1\alpha})) \, ,\\
&= &  o(n^{-1-\frac 1{\alpha\beta}}).
\end{eqnarray*}
\noindent By using \eqref{intZtilde} and Lemma \ref{resteint}, we see that it remains to estimate
$$\int_{|t|\le n^{-\delta + \eta}}
 {\mathbb E}
  \left[ e^{-\vert t\vert^\beta
      \widetilde V_n(A_1+i A_2 \textrm{sgn}(t))}
  {\bf 1}_{\{\sum_y\widetilde N_n(y)\in d_0{\mathbb Z}\}} 
  {\bf 1}_{\{S_n=0\}}{\bf 1}_{\widetilde \Omega_n} \right]\, dt.$$
But, as in the proof of Lemma \ref{le13}, we have
$$\int_{\vert t\vert\le n^{-\delta + \eta}}
  e^{-\vert t\vert^\beta
      \widetilde V_n(A_1+iA_2\textrm{sgn}(t))} \, dt
    = n^{-\delta}\left\{
2\pi\frac{n^\delta}{\widetilde V_n^{\frac 1\beta}}
      f_\beta\left(0\right)
\right\} +o(n^{-\delta}),$$
uniformly on $\widetilde \Omega_n$. We next take the expectation in both sides and  we conclude the proof by using that $\PP(S_n=0) \sim f_\alpha(0)n^{-1/\alpha}$. 
\end{proof}
\noindent The following lemma allows us to get rid of 
${\bf 1}_{ \{\sum_y\widetilde N_n(y)\in d_0{\mathbb Z}\} }$.
\begin{lem}\label{leminter0}
Assume that $d_1$ is a multiple of $d_0$.   
On $\{S_n=0\}$, we have
$$ \sum_y\widetilde N_n(y) \in d_0 {\mathbb Z} \ \ \Leftrightarrow\ \ n\in d_0 {\mathbb Z}.$$
\end{lem}
\begin{proof}
Let $m_n:=\sum_y\widetilde N_n(y)=\sum_{k=1}^n\varepsilon_k$ be the number of 
horizontal moves before time $n$.

\noindent If $S_n=0$, the number $n-m_n$ of vertical moves is necessarily in $d_1{\mathbb Z}$
and so in $d_0{\mathbb Z}$, since $d_1$ is a multiple of $d_0$ by hypothesis. Hence $m_n$ is in $d_0{\mathbb Z}$ if and only if 
$n$ is in $d_0{\mathbb Z}$.
\end{proof}
\noindent We will need the following estimate:
\begin{lem}\label{lemmaone}
Let $V_n:=\sum_{x\in\mathbb Z} N_n(x)^\beta$. Then
$${\mathbb E}\left[|\widetilde{V}_n - p^{\beta} V_n| \ \Big|\ S_n=0 \right]=\O(n^{\delta\beta
    -\frac {\alpha -1}{2\alpha} }).$$
\end{lem}
\begin{proof} 
Set again $m=\lfloor n/2\rfloor$ and $m'=\lceil n/2\rceil$.
By using the inequality $|a^{\beta}-b^{\beta}|\leq \beta |a-b| (a^{\beta-1}+b^{\beta-1})$ 
and the Cauchy--Schwarz inequality, we get
\begin{eqnarray}
\label{in1}
\nonumber {\mathbb E}\left[|\widetilde{V}_n - p^{\beta} V_n| \ \Big| \ S_n=0 \right]&\leq &\beta 
\ {\mathbb E}\left[\sum_{x\in {\mathbb Z}} (\widetilde{N}_n(x)^{\beta-1}+p^{\beta-1} N_n(x)^{\beta-1})^2\ \Big|\ S_n=0\right]^{1/2}\\
&\times  & {\mathbb E}\left[\sum_{x\in {\mathbb Z}} (\widetilde{N}_n(x)-p N_n(x))^2\ \Big|\ S_n=0\right]^{1/2}.
\end{eqnarray} 
We now estimate both expectations in the right hand-side of the above inequality. First note that $N_n(x)=N_m(x)+(N_n(x)-N_m(x))$ and that 
the sequence $((N_n(x)-N_m(x),x\in \ZZ)\mid S_n=0)$ has the same distribution as 
$((N_{m'+1}(-x)-N_1(-x),x\in \ZZ)\mid S_n=0)$. Thus the Markov property gives
\begin{eqnarray*}
& & {\mathbb E}\left[\sum_{x\in {\mathbb Z}} (\widetilde{N}_n(x)^{\beta-1}+p^{\beta-1} N_n(x)^{\beta-1})^2\ \Big|\ S_n=0\right]\leq  4 \sum_{x\in {\mathbb Z}} {\mathbb E}\left[N_{n}(x)^{2(\beta-1)}\ \Big|\ S_n=0\right]\\
&\leq & C\left\{ \sum_{x\in {\mathbb Z}} \sum_{M\in \mathbb Z}  {\mathbb E}\left[N_{m}(x)^{2(\beta-1)}{\bf 1}_{\{S_m=M\}}\right]\frac{\PP(S_{m'}=-M)}{\PP(S_n=0)}\right.\\
 & &+\left.\sum_{x\in {\mathbb Z}} \sum_{M\in \mathbb Z}{\mathbb E}\left[(N_{m'}(x))^{2(\beta-1)}{\bf 1}_{\{S_{m'}=-M\}}\right]\frac{\PP(S_m=M)}{\PP(S_n=0)}\right\},
\end{eqnarray*}
for some constant $C>0$. Since $\sup_M \PP(S_{m'}=-M)/\PP(S_n=0)<+\infty$ and $\sup_M \PP(S_m=M)/\PP(S_n=0)<+\infty$, we get 
$${\mathbb E}\left[\sum_{x\in {\mathbb Z}} (\widetilde{N}_n(x)^{\beta-1}+p^{\beta-1} 
N_n(x)^{\beta-1})^2\ \Big|\ S_n=0\right]\le C \sum_{x\in {\mathbb Z}}  {\mathbb E}\left[N_{m'}(x)^{2(\beta-1)}\right].$$
Then Markov property again and \eqref{Nnynu} show that
\begin{eqnarray}
\label{in2}
\nonumber {\mathbb E}\left[\sum_{x\in {\mathbb Z}} (\widetilde{N}_n(x)^{\beta-1}+p^{\beta-1} N_n(x)^{\beta-1})^2\ \Big|\ S_n=0\right]
&\le & C \EE[R_{m'}] \times \EE\left[N_{m'}(0)^{2(\beta-1)}\right] \\
&=&\O\left(n^{2(\beta-1)(1-\frac 1 \alpha)+\frac 1 \alpha}\right).
\end{eqnarray}
The same argument gives 
\begin{eqnarray*}
& &\sum_{x\in{\mathbb Z}}{\mathbb E}\left[ (\widetilde{N}_n(x)-p N_n(x))^2\ \Big|\ S_n=0\right]\\
&\le & C \left\{\sum_{x\in{\mathbb Z}}{\mathbb E}\left[ (\widetilde{N}_m(x)-p N_m(x))^2\right]+\sum_{x\in{\mathbb Z}}{\mathbb E}\left[ (\widetilde{N}_{m'}(x)-p N_{m'}(x))^2\right]\right\}, 
\end{eqnarray*}
for some constant $C>0$. Then by using \eqref{NtildeN} (note that 
$\bar{\varepsilon}_k$ is centered and independent of $(S_{\ell-1},\varepsilon_\ell,S_{k-1})$ 
if $\ell<k$), we get
\begin{eqnarray}
\label{in3}
\sum_{x\in{\mathbb Z}}{\mathbb E}\left[ (\widetilde{N}_n(x)-p N_n(x))^2\ \Big|\ S_n=0\right]=\O(n).
\end{eqnarray}
\noindent The lemma now follows by combining \eqref{in1}, \eqref{in2} and \eqref{in3}
since $(\beta-1)\left(1-\frac 1\alpha\right)+\frac 1{2\alpha}+\frac 12
   =\delta\beta-\frac{\alpha-1}{2\alpha}$. 
\end{proof}
\begin{lem}\label{lemmatou}
Conditionally to the event $\{S_n=0\}$, the sequence $(V_n/n^{\delta\beta},n\ge 0)$ 
converges in distribution to the random variable $\int_{\mathbb R} (L_1^0(x))^\beta\, dx$.
 \end{lem}
\begin{proof}
According to \cite{NadineClement}, the lemma will essentially follow from the two 
following statements~:
\begin{itemize}
\item[{\bf (RW1)}] The sequence of conditioned processes
$\left(\left(n^{-1/\alpha}S_{\lfloor nt \rfloor}
\left.\right\vert S_{n}=0\right),t\in [0,1]\right)$ 
converges in distribution to the bridge $(\widetilde U^0_t,t\in [0,1])$, as $n\to \infty$.
\item[{\bf (RW2)}]
\begin{itemize}
\item[(i)] $$\sup_y{\mathbb E}\left[N_{n}(y)^2\ \vert\ S_{n}=0\right]=
       \O(n^{2-\frac 2\alpha}).$$
\item[(ii)] There exists a constant $C>0$ such that for every $x,y\in {\mathbb R}$,
$${\mathbb E}\left[ \left.\left( N_{n}\left(\lfloor 
       n^{\frac 1 \alpha}x\rfloor \right) - N_{n}\left(\lfloor 
       n^{\frac 1 \alpha}y\rfloor \right)  \right)^2
       \ \right\vert\ S_n=0\right]\le C n^{2-\frac 2\alpha} |x-y|^{\alpha-1}.$$
\end{itemize}
\end{itemize}
Part {\bf (RW1)} is proved in \cite{Liggett}.

\noindent We prove now {\bf (RW2)} starting with Part $(i)$. By using the same argument as in the proof of Lemma \ref{lemmaone}, we get
$${\mathbb E}\left[N_{n}(y)^2\ \vert\ S_{n}=0\right]\le C (\EE[N_m(y)^2]+\EE[N_{m'+1}(-y)^2]),$$
for some constant $C>0$, with $m$ and $m'$ as in the previous lemma. The desired result now follows from Lemma 1 in \cite{KestenSpitzer}.

\noindent For Part $(ii)$, set $N_n(x,y):=N_n(x)-N_n(y)$. Then use again the argument of the previous lemma, which gives 
$$\EE[N_n(x,y)^2\mid S_n=0] \le C (\EE[N_m(x,y)^2]+\EE[N_{m'+1}(-x,-y)^2]+1),$$
for some constant $C>0$. The result then follows from Lemma 3 in \cite{KestenSpitzer}.

\noindent We can now apply Theorem 4.1 in \cite{NadineClement} in the case when the random scenery is a sequence of i.i.d. random variables with $\beta-$stable distribution and with characteristic function of the form $\theta\mapsto\exp(-c|\theta|^{\beta})$. We deduce that conditionally to $\{S_n=0\}$, 
$$n^{-\delta} \sum_{k=1}^n \xi_{S_k} \mathop{\longrightarrow}_{n\rightarrow\infty}
^{\mathcal{L}} \int_{\mathbb R} L_1^0(x)\, dY_x,$$
where $(Y_x,x\in \RR)$ is a two-sided $\beta-$stable L\'evy process independent of $\widetilde U^0$
and limit in distribution of $\left(n^{-\frac 1\beta}\sum_{k=0}^{\lfloor nx\rfloor}
   \xi_k,\ x\in{\mathbb R}\right)$, when $n\to \infty$. Therefore (see for instance Lemma 5 in \cite{KestenSpitzer}), for every $\theta\in{\mathbb R }$,
$$ {\mathbb E} \left(\exp\left(-c|\theta|^\beta n^{-\delta\beta}V_n\right)\ \Big|\ S_n=0\right) 
\rightarrow {\mathbb E} \left(e^{-c|\theta|^{\beta}\int_{{\mathbb R}} (L_1^0(x))^\beta\, dx}\right)\quad \textrm{when }n\to \infty,
$$
which proves the lemma.
\end{proof}

\begin{lem}\label{lemmatoubis} Conditionally to the event $\{S_n=0\}$, 
the sequence $(n^{\delta\beta}\widetilde V_n^{-1}{\bf 1}_{\widetilde \Omega_n},n\ge 0)$ converges in distribution to the random variable $(p|L^0|_\beta)^{-\beta}$.
 \end{lem}
\begin{proof}
By Lemma \ref{lemmatou}, the sequence 
$(n^{\delta\beta}V_n^{-1}, n \ge 0)$  converges in distribution to
$|L^0|_\beta^{-\beta}$, conditionally to $\{S_n=0\}$.
On the other hand, Lemma \ref{omegasec5} implies that the sequence 
$({\bf 1}_{\widetilde \Omega_n} ,n\ge 0)$ converges in distribution 
to the constant $1$, conditionally to 
$\{S_n=0\}$. 
Hence, 
the sequence $(n^{\delta\beta}V_n^{-1} {\bf 1}_{\widetilde \Omega_n}, n \ge 0)$
 converges in distribution to
$|L^0|_\beta^{-\beta}$, conditionally to $\{S_n=0\}$. 
Next recall that on $\widetilde \Omega_n$, 
$V_n\ge \widetilde V_n \ge cn^{\delta\beta-\gamma}$, for some constant $c>0$. Thus Lemma \ref{lemmaone} gives 
$$\EE\left[\left|\frac{n^{\delta\beta}}{\widetilde V_n}-\frac{n^{\delta\beta}}{p^\beta V_n}
\right|{\bf 1}_{\widetilde \Omega_n} \ \Big|\ S_n=0\right]=
\O\left(n^{-2\delta \beta +2\gamma+2\delta\beta-\frac{\alpha-1}{2\alpha}}\right)
=\O\left(n^{2\gamma-\frac{\alpha-1}{2\alpha}}\right).$$
Therefore, since $\gamma<(\alpha-1)/(4\alpha)$, the left hand side in the above equation converges to $0$ when $n\to \infty$. The lemma follows. 
\end{proof}

\noindent We finally obtain the 

\vspace{0.2cm} 
\noindent \textit{Proof of Theorem \ref{thrwrol}.}
The uniform integrability of the sequence 
$(n^{\delta}\widetilde V_n^{-1/\beta}
{\bf 1}_{\widetilde \Omega_n}, n \ge 0)$ conditionally to $\acc{S_n=0}$ is deduced from Lemma \ref{lem:sec:normeVn} . It
then follows from Lemma \ref{lemmatoubis}  that 
$$\EE\left[\frac{n^{\delta}}{\widetilde V_n^{\frac 1\beta}}{\bf 1}_{\widetilde \Omega_n} \ \Big| \ S_n=0 \right] \to 
p^{-1}\EE[|L^0|_\beta^{-1}] \quad \textrm{when }n\to \infty.$$

\noindent 
The theorem now follows from Lemmas \ref{leminter} and \ref{leminter0}.  

\appendix

\section{Control of the range}
We first gather some known facts about the range $R_n$ of the random walk $(S_n,n\ge 0)$. 
First of all, this walk is transient if, and only if, $\alpha < 1$. 
Moreover, there exists a constant $c>0$ such that 
\begin{equation}
\label{MeanRange}
\EE[R_n] \sim c \left\{ \begin{array}{ll}
				n & \mbox{ if } \alpha < 1 (\mbox{ 
				see \cite{spitzer} p.36}) \, ,
				\\ \frac{n}{\log(n)} & 
				\mbox{ if } \alpha = 1 (\mbox{ 
				see \cite{LGR} Theorem 6.9 p.698}) \, ,
				\\ n^{1/\alpha} &
				\mbox{ if } \alpha > 1 (\mbox{
				see \cite{LGR} Equation (7.a) p.703}) \, .
				\end{array} 
			\right.
\end{equation}
In addition, if $\alpha \leq 1$ (see \cite{spitzer} p.38-40 for $\alpha <1$, and \cite{LGR} Theorem 6.9 for 
$\alpha =1$), then
\begin{equation} 
\label{RangeTransient}
  \frac{R_n}{\EE[R_n]} \rightarrow  1 \mbox{ a.s.}  
\end{equation}
If $\alpha > 1$, it is proved in 
\cite{LGR} Theorem 7.1 p.703, that 
\[ \frac{R_n}{n^{1/\alpha}} \rightarrow  \lambda(U([0,1]))  
\mbox{ in distribution},
\]
where $\lambda$ denotes the Lebesgue measure, and $(U(s), s \in [0,1])$
is an $\alpha$-stable process. In this case, it is also proved in
\cite{LGR} that the constant $c$ appearing in \refeq{MeanRange}
is $\EE\cro{\lambda(U([0,1]))}$,  so that 
\begin{equation}
\label{RangeRec} 
 \frac{R_n}{\EE[R_n]}  \rightarrow  
\frac{\lambda(U([0,1]))}{\EE\cro{\lambda(U([0,1]))}}  
\mbox{ in distribution}.
\end{equation}

\noindent Our aim in this appendix is to prove the following result: 

\begin{lem}\label{lem:omega'_n}
Assume that $\alpha\in (0,2]$. Let $\gamma  \in (0,1/\alpha)$, and set 
\[
\R_n
:= \acc{  
\EE[R_n] n^{-\gamma} \le R_n \le \EE[R_n] n^{\gamma} } \, .
\]
Then there exists a constant $C>0$, such that 
\begin{equation}
\label{ContOmega}
\PP(\R_n)= 1- \O(\exp(-Cn^\gamma)) \, .
\end{equation}
\end{lem}

\begin{proof}
We first prove that for $n$ large enough, 
\begin{equation}
\label{BSRange}
\PP\cro{ R_n \geq \EE[R_n] n^{\gamma}} \leq \exp(-C n^{\gamma}).
\end{equation}
Let us recall that for every $a, b \in \NN$, we have 
\begin{equation} 
\label{SousaddRange}
\PP(R_n \geq a+b) \leq \PP(R_n \geq a) \PP(R_n \geq b) \, .
\end{equation} 
The proof is given for instance in \cite{Chen} 
and goes as follows. Let $\tau:= \inf\acc{k \, :\, R_k \geq a}$. Note that $\tau$ is a stopping time, and that
$R_{\tau} = a$ on $\{\tau < \infty\}$. Moreover,  
\begin{eqnarray*} 
\PP(R_n \geq a + b) & = & \PP (\tau \leq n; R_n \geq a+b)
\\
& = & \sum_{j=1}^n \PP(\tau = j; R_n \geq R_j + b)
\end{eqnarray*} 
Now, for $j \leq n$, $R_n \leq R_j + \# \acc{S_{j+1}, \cdots, S_n}
= R_j + \#\acc{S_{j+1}-S_j, \cdots, S_n-S_j}$.
By independence, we have
\begin{eqnarray*} 
\PP(R_n \geq a + b) & \leq & \sum_{j=1}^n \PP(\tau = j) \PP(R_{n-j} \geq b)
\\
& \leq & \PP(R_n \geq b) \PP(\tau \leq n),
\end{eqnarray*} 
proving \eqref{SousaddRange}. Hence, 
\begin{eqnarray*} 
\PP\left( R_n \geq \EE[R_n] n^{\gamma}\right) 
& \leq & \PP\left({ R_n \geq \floor{3 \EE[R_n]} \floor{\frac{n^{\gamma}}{3}}}\right)
\leq \PP\left({ R_n \geq \floor{3 \EE[R_n]}}\right)^{\floor{\frac{n^{\gamma}}{3}}}
\\
& \leq & \pare{\frac{\EE[R_n]}{\floor{3 \EE[R_n]}}}^{\floor{\frac{n^{\gamma}}{3}}}
\leq \pare{\frac{\EE[R_n]}{3 \EE[R_n]-1}}^{\floor{\frac{n^{\gamma}}{3}}}
\leq \pare{\frac{1}{2}}^{\floor{\frac{n^{\gamma}}{3}}}
\, .
\end{eqnarray*} 
This finishes the proof of \refeq{BSRange}. It remains now to prove that for $n$ large enough,
\begin{equation} 
\label{BIRange}
\PP\left({R_n \leq \EE[R_n] n^{-\gamma}}\right) \leq \exp(-C n^{\gamma}).
\end{equation} 
To this end, let $I_1,\cdots,I_N$ be disjoint subsequent intervals
of $\acc{0,\cdots,n}$, of the same length $l_n$ depending on $n$, so that
$l_n \gg 1$ and $N=\floor{n/l_n}$. Note that 
\[ R_n \geq \max_{j=1}^N\pare{\#\acc{S_k, k \in I_j}} \, , 
\]
and that the random variables $(\#\acc{S_k, k \in I_j}, 1\leq j \leq N)$ 
are i.i.d with the same law as $R_{l_n}$. Hence 
\begin{equation} 
\label{BIRange.2}
\PP\left({R_n \leq \EE[R_n] n^{-\gamma}}\right) 
\leq  \PP\left({\max_{j=1}^N\pare{\#\acc{S_k, k \in I_j}} \leq E(R_n) n^{-\gamma}}\right)
= \PP \left({R_{l_n} \leq \EE[R_n] n^{-\gamma}}\right)^{N}.
\end{equation} 
Choose now $l_n$ such that $\EE[R_{l_n}] \sim  3 \EE[R_n] n^{-\gamma}$.
By \refeq{MeanRange}, this gives
\[
l_n \sim \left\{ \begin{array}{ll}
		3 n^{1-\gamma} & \mbox{ if } \alpha < 1 
		\\
		3 (1-\gamma) n^{1-\gamma} & \mbox{ if } \alpha =1 
		\\
		3^{\alpha} n^{1-\alpha \gamma} &
		\mbox{ if } \alpha > 1, 
		\end{array} 
		\right.
\] 
so that 
\begin{eqnarray}
\label{N}
N \sim  \left\{ \begin{array}{ll}
		\frac{1}{3}  n^{\gamma} & \mbox{ if } \alpha < 1 
		\\
		\frac{1}{3(1-\gamma)} n^{\gamma} & \mbox{ if } \alpha =1 
		\\
		\frac{1}{3^{\alpha}} n^{\alpha \gamma} &
		\mbox{ if } \alpha > 1. 
		\end{array} 
		\right.
\end{eqnarray}
For $n$ large enough, $\EE[R_{l_n}] \geq 2 \EE[R_n] n^{-\gamma}$, and 
it follows from \refeq{BIRange.2} that 
\begin{eqnarray}
\label{eqRn}
\PP\left({R_n \leq \EE[R_n] n^{-\gamma}}\right) \leq  
 \PP \left({R_{l_n} \leq \frac{\EE[R_{l_n}]}{2}}\right)^N \, .
\end{eqnarray}
For $\alpha \leq 1$, $\PP \left({R_{l_n} \leq \frac{\EE[R_{l_n}]}{2}}\right)$ 
tends to zero by \refeq{RangeTransient}. By \refeq{RangeRec}, for $\alpha >1$, 
we have
$$\limsup_{n \rightarrow \infty} \PP \left({R_{l_n} \leq \frac{\EE[R_{l_n}]}{2}}
\right)
\leq  \PP \cro{({\lambda(U([0,1])) \leq \frac{1}{2}
 \EE\cro{\lambda(U([0,1]))}}}
< 1,$$
since a.s. $\lambda(U([0,1])) > 0$. 
In any case 
there exists $p< 1$, such
that for all $\gamma \in (0,1/\alpha)$, and for $n$ large enough, 
$$\PP \left({R_{l_n} \leq \frac{\EE[R_{l_n}]}{2}}\right) \leq p.$$
Together with \eqref{eqRn} and \eqref{N}, this proves \refeq{BIRange} and the lemma. 
\end{proof}

\end{document}